\documentclass[11pt,a4paper]{amsart}
\usepackage{amsmath,amssymb,color,multicol,setspace}
\usepackage[utf8,utf8x]{inputenc}

\usepackage{longtable}
\usepackage{xy,amscd} 

\usepackage{mathdots}
\usepackage{epsfig}
\usepackage{rotating}
\usepackage{xspace}
\usepackage{float}
\usepackage{enumitem} 
\usepackage{comment}
\xyoption{all}
\usepackage{tikz}
\usetikzlibrary{arrows,decorations.pathmorphing,decorations.pathreplacing,positioning,shapes.geometric,shapes.misc,decorations.markings,decorations.fractals,calc,patterns}

\newcommand{\red}{\color{red}}

\newtheorem{Lemma}{Lemma}[section]
\newtheorem{Theorem}[Lemma]{Theorem}
\newtheorem{Proposition}[Lemma]{Proposition}
\newtheorem{Corollary}[Lemma]{Corollary}

\theoremstyle{definition}
\newtheorem{Definition}[Lemma]{Definition}

\newtheorem{Remark}[Lemma]{Remark}

\newtheorem{Example}[Lemma]{Example}

\newtheorem{Algorithm}[Lemma]{Algorithm}

\numberwithin{equation}{section}

\newcommand{\id }{\mathrm{id}}

\DeclareMathOperator{\Hom}{Hom}

\DeclareMathOperator{\Dih}{Dih}

\newcommand{\re }{^\mathrm{re}}
\newcommand{\rer }[1]{(R\re)^{#1}}
\newcommand{\rsC }{\mathcal{R}}
\newcommand{\rfl }{\rho }
\newcommand{\Cm }{C}

\newcommand{\Cc }{\mathcal{C}}
\newcommand{\cm }{c}
\newcommand{\s }{\sigma }
\newcommand{\Wg }{\mathcal{W}}
\newcommand{\Ob }{\mathrm{Ob}}

\newcommand{\qs}{q}

\newcommand{\Ri}{D_p}

\title[Finite and infinite frieze patterns from $p$-angulations]
{Finite and infinite frieze patterns from $p$-angulations and a generalization of Weyl
groupoids}

\author{Michael Cuntz}
\address{Michael Cuntz, Leibniz Universit\"at Hannover,
Institut f\"ur Algebra, Zahlentheorie und Diskrete Mathematik,
Fakult\"at f\"ur Mathematik und Physik,
Welfengarten 1,
30167 Hannover, Germany}
\email{cuntz@math.uni-hannover.de}
\urladdr{http://www.iazd.uni-hannover.de/de/cuntz}

\author{Thorsten Holm}
\address{Thorsten Holm, Leibniz Universit\"at Hannover,
Institut f\"ur Algebra, Zahlentheorie und Diskrete Mathematik,
Fakult\"at f\"ur Mathematik und Physik,
Welfengarten 1,
30167 Hannover, Germany}
\email{holm@math.uni-hannover.de}
\urladdr{http://www.iazd.uni-hannover.de/de/holm}

\author{Peter J{\o}rgensen}
\address{Peter J{\o}rgensen, 
Department of Mathematics,
Aarhus University,
Ny Munkegade 118,
8000 Aarhus C, Denmark}
\email{peter.jorgensen@math.au.dk}
\urladdr{https://sites.google.com/view/peterjorgensen}

\keywords{Cartan matrix, Cartan graph, 
Chebyshev polynomial,
Dissection, Frieze pattern, Infinite frieze pattern,
Polygon, Weyl groupoid}

\subjclass[2020]{05E99, 13F60, 17B22, 
18B40, 51M20, 52B45}

\begin{document}

\begin{abstract}
A classic result of Conway and Coxeter on frieze patterns has recently been generalized 
by establishing for every integer $p\ge 3$ 
a bijection between $p$-angulations of regular polygons and a new class of frieze patterns, called frieze
patterns of type $\Lambda_p$.   
One of the features of Conway-Coxeter theory is a combinatorial procedure to obtain from the 
triangulation all entries of the corresponding frieze pattern. We first show how to extend this to dissections 
and to frieze patterns of type $\Lambda_p$, that is, we present a combinatorial algorithm, involving
Chebyshev polynomials, for obtaining
from a dissection all entries of the corresponding frieze pattern.  
As an application we obtain an explicit characterisation of frieze patterns of types $\Lambda_4$ and
$\Lambda_6$ in terms of all entries (not only the quiddity cycle), partially answering a question 
in an earlier paper. 
We then study infinite frieze patterns of type $\Lambda_p$, which appeared in a preprint by Banaian and Chen, 
generalizing the infinite frieze patterns 
of positive integers studied by Baur, Parsons and Tschabold. As our main result we obtain a 
combinatorial model for infinite frieze patterns of type $\Lambda_p$, that is, we show that these are 
in bijection with certain $p$-angulations of an infinite strip. This extends results by 
Baur, Parsons and Tschabold from $p=3$ to arbitrary $p\ge 3$, but we provide new insight even in the classic 
case in that our combinatorial model gives a bijection and also a more transparent constructive approach.
It is known that infinite frieze patterns 
of positive integers naturally appear in the
context of Weyl groupoids. In the final section 
we extend this connection to infinite frieze 
patterns of type $\Lambda_p$ for any $p\ge 3$ 
by introducing a generalization of Cartan 
graphs and Weyl 
groupoids called reflection groupoids. 
We show that, up to equivalence, there is a 1-1 correspondence 
between connected simply connected Cartan graphs of 
type $\Lambda_p$ of rank two with infinitely many vertices permitting a root system and 
infinite frieze patterns of type $\Lambda_p$.
\end{abstract}

\maketitle

\section{Introduction}

Triangulations and, more generally, dissections of polygons are ubiquitous objects in mathematics. One instance 
where triangulations appear is the theory of frieze patterns, where a classic result of Conway and Coxeter 
\cite{CC73} yields a bijection between triangulations of regular polygons and frieze patterns over positive
integers. Only recently, 
this beautiful theory could be generalized from triangulations to dissections of polygons: In \cite{HJ17} we show that
for every natural number $p\ge 3$ there is a bijection between $p$-angulations of polygons and a new class of frieze
patterns,
called frieze patterns of type $\Lambda_p$. These frieze patterns have entries in the ring of integers of the algebraic number field
$\mathbb{Q}(2\cos(\frac{\pi}{p}))$. 
In our paper \cite{HJ17} we actually show more: for every dissection (not only for $p$-angulations) one can 
naturally associate a frieze pattern. Namely, given a dissection, say with subpolygons of sizes $q_1,\ldots,q_r$, 
we label each $q_i$-gon by $\lambda_{q_i}=2\cos(\frac{\pi}{q_i})$. At each vertex of the polygon we then form the 
sum of the labels of the subpolygons attached to this vertex. Then the series of these sums forms the first
non-trivial row of the corresponding frieze pattern. This first row actually determines the entire frieze pattern
since the frieze patterns are shown to have no entries equal to zero. 

The combinatorial procedure described above only determines the entries in the first non-trivial row. For the classic
Conway-Coxeter frieze patterns there is also a nice combinatorial method to determine all entries of the corresponding
frieze patterns directly from the triangulation, see for instance \cite{BCI74}. 
As a useful tool we establish a combinatorial algorithm for producing all entries of the frieze 
patterns corresponding to arbitrary dissections. This is done in Section \ref{sec:algorithm}, see in particular
Algorithm \ref{alg:frieze}. This procedure needs a variation of the well-known Chebyshev polynomials 
of the second kind. We introduce these polynomials in Section \ref{sec:chebyshev} and recall some of their
fundamental properties needed later in the paper. 

In Section \ref{sec:lambda46} we apply the combinatorial algorithm from Section \ref{sec:algorithm}
to give a partial answer to a fundamental question posed in \cite{HJ17}, namely whether there is 
a useful characterisation of all entries of frieze patterns of type $\Lambda_p$. In Theorem \ref{thm:lambda46} 
we give such an answer for the cases $p=4$ and $p=6$, extending earlier work by Andritsch \cite{A20}. 
\medskip

\noindent
{\bf Theorem \ref{thm:lambda46}.} 
{\em Let $\mathcal{F}$ be a frieze pattern. Then the following statements are equivalent:
\begin{enumerate}
\item[{(a)}] $\mathcal{F}$ is a frieze pattern of type $\Lambda_4$ or $\Lambda_6$.
\item[{(b)}] The entries of $\mathcal{F}$ satisfy the following two conditions. 
\begin{itemize}
\item[{(i)}] The odd-numbered rows consist of positive integers.
\item[{(ii)}] The even-numbered rows consist of positive integral multiples of $\sqrt{2}$ (for type $\Lambda_4$)
or $\sqrt{3}$ (for type $\Lambda_6$), respectively.
\end{itemize}
\end{enumerate}
}

\medskip

In Sections \ref{sec:infiniteLambdap} and \ref{sec:combmodel} we introduce and study infinite
frieze patterns of type $\Lambda_p$. Infinite frieze patterns of positive integers have been introduced 
by Tschabold \cite{T15} and studied further by Baur, Parsons, Tschabold \cite{BPT16}. Infinite frieze patterns 
of type $\Lambda_p$ appeared already in the preprint by Banaian and Chen \cite{BC21}, where 
more generally infinite frieze patterns for dissections have been considered. In all this earlier work,
the most important goal has not been achieved, namely finding a bijection between these infinite 
frieze patterns and a combinatorial model (in the spirit 
of the classic Conway-Coxeter bijection between frieze patterns of positive integers and triangulations 
of polygons). Even for the classic case $p=3$, the model presented in \cite{BPT16} does not give 
a bijection. Our main result in Section \ref{sec:combmodel} provides such a bijection for 
infinite frieze patterns of type $\Lambda_p$, for any $p\ge 3$. 
\medskip

\noindent
{\bf Theorem \ref{thm:bijection}.}
{\em Let $p\ge 3$. There is a bijection between infinite frieze patterns of type $\Lambda_p$ and
locally finite 
$p$-angulations of infinite strips (with vertices $\mathbb{Z}$ on the lower line and vertices
from a subset $S\subseteq \mathbb{Z}$ on the upper line), up to translational renumbering 
of the vertices on the upper line. All entries of the infinite frieze pattern corresponding to a
$p$-angulation can be computed with Algorithm \ref{alg:frieze}. 
}
 \medskip
 
This main theorem and the results obtained along the way generalize results on infinite frieze
patterns over positive integers to infinite frieze patterns of type $\Lambda_p$ 
and provide new insight even in the classic case $p=3$. In particular, starting from an infinite 
frieze pattern of type $\Lambda_p$ we obtain an explicit
construction for the corresponding $p$-angulation of an infinite strip.  

As one consequence we obtain bounds for the entries of infinite frieze patterns of type $\Lambda_p$
(for finite frieze patterns of type $\Lambda_p$ this is a consequence of Algorithm 4.1, see Corollary 
\ref{cor:ge1}). 
\medskip

\noindent
{\bf Corollary \ref{cor:infge1}.}
{\em Let $\mathcal{F}$ be an infinite frieze pattern of type $\Lambda_p$. Then the entries of $\mathcal{F}$
belong to the set $\{1\}\cup [\lambda_p,\infty)$. 
}
\medskip

Note that this is obvious for $p=3$ (as $\lambda_3=1$), but a non-trivial statement for $p>3$. 
\medskip

As another consequence we get the analog of Theorem \ref{thm:lambda46} also for infinite frieze patterns
of type $\Lambda_4$ and $\Lambda_6$ in Corollary \ref{cor:inflambda46}.
\medskip

Weyl groupoids were first introduced under the name of arithmetic root systems by Heckenberger \cite{p-H-06}, \cite{p-H-09} in order to classify finite dimensional Nichols algebras of diagonal type. Understanding Nichols algebras is a step within the classification of pointed Hopf algebras proposed by Andruskiewitsch and Schneider \cite{AS02}.
After a system of axioms for Weyl groupoids was proposed in \cite{p-HY08} and \cite{p-CH09a}, Heckenberger and the first author gave a complete classification of finite Weyl groupoids in all ranks \cite{CH15}.

The definition of Weyl groupoids uses 
Cartan graphs, and one feature of these is
that Cartan matrices (known from Lie theory)
are attached to the vertices. 
It is known that Weyl groupoids of rank 2 
(admitting a root system) are closely related to 
finite and infinite frieze patterns of positive 
integers. In the final section of this article 
we generalize this to frieze patterns of type
$\Lambda_p$ for any integer $p\ge 3$. To this end,
we introduce generalized Cartan matrices of type
$\Lambda_p$ and from this Cartan graphs of
type $\Lambda_p$ (see Definition \ref{def:cartangraph}). This leads us 
to the new notion of reflection groupoids, 
which generalizes Weyl groupoids from $p=3$ to
arbitrary $p\ge 3$. As our main result in 
Section \ref{sec:groupoids} get a bijection 
between infinite frieze patterns of type $\Lambda_p$
and certain Cartan graphs (or reflection groupoids):

\medskip

\noindent
{\bf Theorem \ref{thm:cartanfriezep}.}
{\em 
There is a 1-1 correspondence between connected simply connected Cartan graphs of type $\Lambda_p$ of rank two with infinitely many 
vertices permitting a root system
up to equivalence and (tame) infinite frieze patterns of type $\Lambda_p$ up to equivalence.
}

\section{Frieze patterns} \label{sec:friezes}

Frieze patterns have been introduced by Coxeter \cite{Cox71}. We recall the definition here. 

\begin{Definition} \label{def:frieze}
Let $m\ge 0$ be an integer. 
A {\em Coxeter frieze pattern} or just {\em frieze pattern}
of height $m$ consists of $m+2$ infinite horizontal rows of positive real numbers, with an
offset between neighbouring rows. The top and bottom rows consist only of 1's. The other rows contain
positive real numbers such that each 'diamond' of four neighbouring numbers\,\,\,\,
$\begin{matrix} & b & \\ a & & d \\ & c & \end{matrix}$\,\,\,
satisfies the 'unimodular rule' $ad-bc=1$. 

The second row is called the {\em quiddity row} of the frieze pattern. 
\smallskip

It is convenient to label the entries in a frieze pattern of height $m$ in the following way: 
{\small
$$\begin{array}{ccccccccccccccccc}
& & & & 1 & & 1 & & {\red 1} & & 1 & & 1 & & & & \\
& & & & & f_{-1,m} & & {\red f_{0,m+1}} & & {\red f_{1,m+2}} & & f_{2,m+3} & & & \\
\ldots & & & & \iddots & & {\red \iddots} & & {\red \iddots}  & & {\red \ddots} & & \ddots & & & & \ldots \\
 & & & \iddots & & {\red \iddots} & & {\red \iddots}  & & & & {\red \ddots} & & \ddots & & &  \\
& & f_{-1,2} & & {\red f_{0,3}} & & {\red f_{1,4}} & &  & & & & {\red f_{m-1,m+2}} & & f_{m,m+3} & & \\
& f_{-1,1} & & {\red f_{0,2}} & & {\red f_{1,3}} & & &  & & & & & {\red f_{m,m+2}} & & f_{m+1,m+3} & \\
1 & & {\red 1} & & {\red 1} & & {\red 1} & {\red \ldots} & & {\red \ldots} & & {\red \ldots} & {\red 1} & & {\red 1} & & 1\\
\end{array}
$$
}
\end{Definition}

\begin{Remark} \label{rem:frieze}
Frieze patterns can be defined with more general entries, not only for positive real numbers, but the above definition 
covers all the frieze patterns dealt with in this paper. Since the entries of our frieze patterns are non-zero, the quiddity 
row uniquely determines the entire frieze pattern, by repeated application of the unimodular rule. Frieze patterns with
non-zero entries have a glide symmetry (see \cite[Equation (6.1)]{Cox71}), 
a fundamental domain is given by a triangular region, see the 
red area in the above figure in Definition \ref{def:frieze}. 
\end{Remark}

\begin{Example} \label{ex:friezes}
We give some examples of frieze patterns of height 2. (We show a finite part, the frieze patterns
are repeated periodically with period 5.)

The first one is a frieze pattern with (positive) integral entries. 
\begin{figure}[H]
\begingroup
\[
  \vcenter{
  \xymatrix @-0.5pc @!0 {
& 1 & & 1 & & 1 & & 1 & & 1 & & 1 & & 1 \\
2 & & 1 & & 3 & & 1 & & 2 & & 2 & & 1 & & 3 \\
& 1 & & 2 & & 2 & & 1 & & 3 & & 1 & & 2 \\
1 & & 1 & & 1 & & 1 & & 1 & & 1 & & 1 & & 1 \\
        }
   }
\]
\endgroup
\end{figure}
The second example has entries in $\mathbb{Z}[\sqrt{2}]$. 
\begin{figure}[H]
\begingroup
\[
  \vcenter{
  \xymatrix @-0.1pc @!0 {
    & ~1~ & & 1 & & ~1~ & & 1 & & ~1~ & & 1 & & ~1~ \\
~\sqrt{2}~ & & ~\sqrt{2}~ & & ~1+\sqrt{2}~ & & ~1~ & & ~1+\sqrt{2}~ & & ~\sqrt{2}~ & & ~\sqrt{2}~ & & ~1+\sqrt{2}~ \\
& 1 & & 1+\sqrt{2} & & \sqrt{2} & & \sqrt{2} & & 1+\sqrt{2} & & 1 & & 1+\sqrt{2} \\
1 & & 1 & & 1 & & 1 & & 1 & & 1 & & 1 & & 1 \\
        }
   }
\]
\endgroup
\end{figure}
\end{Example}

A classic result of Conway and Coxeter \cite{CC73} shows that, for any $n\ge 3$, there is a bijection between triangulations
of a regular $n$-gon and frieze patterns of height $n-3$ with positive integral entries. Moreover, given a triangulation, the
quiddity row of the corresponding frieze pattern is obtained by counting at each vertex of the polygon
the number of triangles attached to the vertex.  
\smallskip

In recent work, we were able to generalize the classic Conway-Coxeter theory from triangulations to 
$p$-angulations for any $p\ge 3$ \cite{HJ17}. We briefly recall the relevant definitions and results here. 

For any natural number $p\ge 2$ let $\lambda_p=2\cos(\frac{\pi}{p})$. 
In particular, $\lambda_2=0$, $\lambda_3=1$, $\lambda_4=\sqrt{2}$,
$\lambda_5 =\frac{1+\sqrt{5}}{2}$, the golden ratio, $\lambda_6=\sqrt{3}$, etc. 

\begin{Definition}(\cite[Definition 0.2]{HJ17}) \label{def:lambdap}
Let $p\ge 3$ be an integer. A frieze pattern is {\em of type $\Lambda_p$} if the quiddity row consists of 
(necessarily positive) integral multiples of $\lambda_p$.
\end{Definition}

Note that the classic Conway-Coxeter frieze patterns corresponding to triangulations are precisely the 
frieze patterns of type $\Lambda_3$. 
Our generalization of the Conway-Coxeter result reads as follows.

\begin{Theorem}(\cite[Theorem A]{HJ17}) \label{thm:HJ-ThmA}
Let $p\ge 3$ and $n\ge 3$ be integers. There is a bijection between
$p$-angulations of a regular $n$-gon and frieze patterns of type $\Lambda_p$ with height $n-3$. 
\end{Theorem}

Moreover, given a $p$-angulation of a regular $n$-gon, there is again a simple method to determine the quiddity row of the 
corresponding frieze pattern. Namely, each $p$-angle in the $p$-angulation is labelled by $\lambda_p$. 
For each vertex of the $n$-gon, compute the sum of the labels of the attached $p$-angles. The sequence of these
sums (repeated periodically) gives the quiddity row of the corresponding frieze pattern.

\begin{Example}
We consider a quadrangulation of a hexagon. The above procedure leads to the quiddity row
$(\ldots,\sqrt{2},\sqrt{2},2\sqrt{2}, \sqrt{2},\sqrt{2},2\sqrt{2},\ldots)$ and the corresponding frieze pattern 
of type $\Lambda_4$ has the form
$$\begin{matrix} 
& 1 & & 1 & & {\red 1} & & 1 & & 1 & & 1 & & 1 & & 1 & \\
& & \sqrt{2} & & {\red \sqrt{2}} & & {\red 2\sqrt{2}} & & \sqrt{2} & & \sqrt{2} & & 2\sqrt{2} & & \sqrt{2} &  & \\
~\ldots~ & ~3~ & & ~{\red 1}~ & & ~{\red 3}~ & & ~{\red 3}~ & & ~1~ & & ~3~ & & ~3~ & & ~1~ & ~\ldots~ \\
& & {\red \sqrt{2}} & & {\red \sqrt{2}} & & {\red 2\sqrt{2}} & & {\red \sqrt{2}} & & \sqrt{2} & & 2\sqrt{2} & & \sqrt{2} &  &\\
& {\red 1} & & {\red 1} & & {\red 1} & & {\red 1} & & {\red 1} & & 1 & & 1 & & 1  &\\
\end{matrix}
$$
where we have again highlighted a fundamental region in red.
\end{Example}

The results in \cite{HJ17} even yield more general frieze patterns than the ones in Theorem \ref{thm:HJ-ThmA}. We can 
associate a frieze pattern to an arbitrary polygon dissection, not only to $p$-angulations, see \cite[Construction 0.3]{HJ17}. 
Let $\mathcal{T}$ 
be a dissection of a regular $n$-gon $P$, with vertices denoted $0,1,2,\ldots,n-1$ in counterclockwise order. 
We label each subpolygon in this dissection with $\lambda_q$, where $q$
is the number of vertices of the subpolygon. Note that each triangle is labelled by 1, each quadrangle by 
$\sqrt{2}$, each pentagon by $\frac{1+\sqrt{5}}{2}$ etc. 
Then for each vertex $i$ of $P$ we add the labels of all subpolygons adjacent to $i$, and call this sum
$c_i$. Then \cite[Theorem B]{HJ17} states that $(c_0,c_1,\ldots,c_{n-1})$, repeated periodically, is the quiddity
row of a frieze pattern $F_{\mathcal{T}}$ (of height $n-3$) corresponding to the dissection $\mathcal{T}$. The entries in this frieze pattern 
are in the ring of algebraic integers $\mathcal{O}_K$
of the field 
$K=\mathbb{Q}(\lambda_{p_1},\ldots,\lambda_{p_s})$ where $p_1,\ldots,p_s$ are the sizes of the subpolygons of the dissection $\mathcal{T}$. Moreover, all entries in these frieze patterns
are $\ge 1$, see \cite[Lemma 3.3]{HJ17}. 

\begin{Example} \label{ex:frieze46}
(1) The second frieze pattern in Example \ref{ex:friezes} is of the form $F_{\mathcal{T}}$ where $\mathcal{T}$
is a dissection of a pentagon into a triangle and a quadrangle. 
\smallskip

\noindent
(2) 
The following is a frieze pattern over $\mathbb{Z}[\sqrt{2},\sqrt{3}]$, the ring of integers of the algebraic number field 
$\mathbb{Q}(\lambda_4,\lambda_6)=\mathbb{Q}(\sqrt{2},\sqrt{3})$. It comes from a dissection of an octagon
into a quadrangle and a hexagon, by the above construction from \cite{HJ17}.

\begin{figure}[H]
\begingroup
\[
  \vcenter{
  \renewcommand{\objectstyle}{\scriptstyle}
  \xymatrix @-0.0pc @!0 {
     1 && 1 && 1 && {\red 1} && 1 && 1 && 1 && 1 && 1  \\    
    & \sqrt{3} && \sqrt{3} && {\red \sqrt{3}} && {\red \sqrt{2}+\sqrt{3}} && \sqrt{2} && \sqrt{2} && \sqrt{2}+
   \sqrt{3} && \sqrt{3} & \\
     2 && 2 && {\red 2} && {\red 2+\sqrt{2}\sqrt{3}} && {\red 1+\sqrt{2}\sqrt{3}} && 1 && 1+\sqrt{2}\sqrt{3} && 2+\sqrt{2}\sqrt{3} && 2\\
    & \sqrt{3} && {\red \sqrt{3}} && {\red 2\sqrt{2}+\sqrt{3}} && {\red 2\sqrt{2}+\sqrt{3}} && {\red \sqrt{3}} && \sqrt{3} && 2\sqrt{2}+\sqrt{3} && 2\sqrt{2}+\sqrt{3} \\
     1+\sqrt{2}\sqrt{3} && {\red 1} && {\red 1+\sqrt{2}\sqrt{3}} && {\red 2+\sqrt{2}\sqrt{3}} && {\red 2} && {\red 2} && 2 && 2+\sqrt{2}\sqrt{3} && 1+\sqrt{2}\sqrt{3} \\   
   & {\red \sqrt{2}} && {\red \sqrt{2}} && {\red \sqrt{2}+\sqrt{3}} && {\red\sqrt{3}} && {\red \sqrt{3}} && {\red \sqrt{3}} && \sqrt{3} && \sqrt{2}+\sqrt{3} &\\   
    {\red 1} && {\red 1} && {\red 1} && {\red 1} && {\red 1} && {\red 1} && {\red 1} && 1 && 1 \\   
        }
   }
\]
\endgroup
\label{fig:sqrt3_frieze}
\end{figure}

\end{Example}

\section{Chebyshev polynomials} \label{sec:chebyshev}

We consider in this section a series of polynomials and review some of their fundamental properties. 
These polynomials are closely related to the famous Chebyshev polynomials of the second kind. 
The relevance of these polynomials in the context of frieze patterns has been observed before, see 
the paper by Dupont \cite{Dup12} for more details. The following definition also appears in 
\cite[Definition 6]{BC21}.

\begin{Definition} \label{def:chebyshev}
We define recursively the following polynomials over the real numbers:
$$V_{-1}(x)=0,\,V_0(x)=1,\mbox{~~and~~}
V_{n+1}(x) = x\cdot V_n(x) - V_{n-1}(x)\mbox{~~for $n\ge 0$}.
$$
\end{Definition}

For instance, we have $V_1(x)=x$, $V_2(x)=x^2-1$, $V_3(x)=x^3-2x$, $V_4(x)=x^4-3x^2+1$,
$V_5(x)= x^5-4x^3+x$ etc.

\begin{Remark} \label{rem:roots}
We have $V_n(x) = U_n(x/2)$ where $U_n$ are the Chebyshev polynomials of the second kind.
Moreover, from a well-known formula for the roots of the Chebyshev polynomials one gets
$$V_n(x) = \prod_{k=1}^n \left(x-2\cos\left(\frac{\pi}{n+1}\cdot k\right)\right)
$$
for every $n\in \mathbb{N}$. 
\end{Remark}

\begin{Remark} \label{rem:Euclidlength}
For $p\ge 3$ let $P$ be a regular $p$-gon with Euclidean side lengths $1=V_0(\lambda_p)$. 
By elementary trigonometry the Euclidean length of any diagonal of combinatorial length 2 is 
$\lambda_p=V_1(\lambda_p)$. Inductively, now consider diagonals of combinatorial lengths 
$i-1,i,i+1$. If the Euclidean lengths of the former two are $V_{i-1}(\lambda_p)$ and $V_i(\lambda_p)$,
then the Euclidean length $\alpha$ of the diagonal with combinatorial length $i+1$ satisfies by Ptolemy's theorem that 
$$\lambda_p V_i(\lambda_p) = V_{i-1}(\lambda_p) \cdot 1 + \alpha\cdot 1 =
 V_{i-1}(\lambda_p) + \alpha.
 $$
 Hence we get
 $$\alpha = \lambda_p V_i(\lambda_p) - V_{i-1}(\lambda_p) = V_{i+1}(\lambda_p).
 $$
 It follows that all values $V_1(\lambda_p),V_2(\lambda_p),\ldots,V_{p-3}(\lambda_p)$ are
 greater than or equal to $\lambda_p$, because these are Euclidean lengths of diagonals of
 combinatorial lengths $\ge 2$. 
\end{Remark}

We summarize some known facts about the polynomials $V_n$ which will be needed in this paper. These all follow from 
well-known properties of Chebyshev polynomials. For completeness we include a proof here.

\begin{Proposition} \label{prop:chebyshev}
With the above notation the following holds.
\begin{enumerate}
\item[{(i)}] $V_n(x)^2 - V_{n-1}(x)V_{n+1}(x) = 1$ for all $n\ge 0$. 
\item[{(ii)}] $V_n(\lambda_{n+1})=0$ for all $n\ge 1$.
\item[{(iii)}] $V_n(\lambda_n)=-1$ for all $n\ge 2$.
\item[{(iv)}] $V_{n}(\lambda_{n+2})=1$ for all $n\ge 0$. 
\end{enumerate}
\end{Proposition}

\begin{proof}
(i) We prove this by induction. 
For $n=0$ we have by Definition \ref{def:chebyshev} 
$$V_0(x)^2- V_{-1}(x)V_1(x) = 1 - 0\cdot x = 1,$$
as claimed. Now let $n>0$. Then we use the recursive formula from Definition \ref{def:chebyshev} and the
induction hypothesis and obtain that
\begin{eqnarray*}
V_n(x)^2 - V_{n-1}(x)V_{n+1}(x) & = & V_n(x)^2 - V_{n-1}(xV_n(x)-V_{n-1}(x)) \\
& = & V_n(x)^2 - xV_{n-1}(x)V_{n}(x) + V_{n-1}(x)^2 \\
& = & V_n(x)^2 - xV_{n-1}(x)V_{n}(x) + (1+V_{n-2}(x) V_{n}(x)) \\
& = & V_n(x) (xV_{n-1}(x)-V_{n-2}(x)) - xV_{n-1}(x)V_{n}(x) + 1+V_{n-2}(x) V_{n}(x) \\
& = & 1
\end{eqnarray*}
and this proves the inductive step.
\smallskip

\noindent
(ii) From the formula in Remark \ref{rem:roots} it is clear that $\lambda_{n+1}=2\cos(\frac{\pi}{n+1})$ is a root of
$V_n(x)$ for $n\ge 1$. 
\smallskip

\noindent
(iii) Evaluating the formula in (i) at $x=\lambda_n$ and using part (ii) yields
that $V_n(\lambda_n)^2 = 1$ for $n\ge 2$, that is, $V_n(\lambda_n)\in \{\pm 1\}$. 
For determining the correct sign, we consider the factorization
$$V_n(x) = \prod_{k=1}^n \left(x- 2\cos\left(\frac{k\pi}{n+1}\right)\right)
$$
from Remark \ref{rem:roots}.
We evaluate at $\lambda_n$ and consider the sign of each factor. For $k=1$, the factor
$\lambda_n - 2\cos(\frac{\pi}{n+1})$ is negative. We claim that for all $k=2,\ldots,n$
the corresponding factor is positive, and this proves part (iii). So let $2\le k\le n$. 
Since $n\ge 2$ we have $kn> n+1 > n$, and thus $\frac{k\pi}{n+1} > \frac{\pi}{n} > \frac{\pi}{n+1}$. 
Since the cosine is strictly decreasing in the interval $[0,\pi]$ we deduce that 
$$2\cos\left(\frac{k\pi}{n+1}\right) < \lambda_n < 2\cos\left(\frac{\pi}{n+1}\right)
$$
and the claim follows.
\smallskip

\noindent
(iv) For any $n\ge 0$ we get the following series of equations by using parts (i), (ii)
and (iii): 
\begin{eqnarray*}
1 & = & V_{n+1}(\lambda_{n+2})^2 - V_n(\lambda_{n+2})V_{n+2}(\lambda_{n+2}) 
 \,\, = \,\, 0 - V_n(\lambda_{n+2})V_{n+2}(\lambda_{n+2}) \\
& = & -V_n(\lambda_{n+2})\cdot (-1) \,\, = \,\, V_n(\lambda_{n+2})
\end{eqnarray*}
proving part (iv).
\end{proof}

As a consequence we get the following symmetry property for our Chebyshev polynomials.

\begin{Proposition} \label{prop:symmetry}
For every integer $q\ge 2$ we have
$$V_j(\lambda_q) = V_{(q-2)-j}(\lambda_q)\mbox{~~~~~~~for all $j=0,1,\ldots,q-2$}.
$$
\end{Proposition}

\begin{proof}
We prove the statement by induction on $j$. 

For $j=0$ we have
$V_0(\lambda_q) = 1 = V_{q-2}(\lambda_q)$
by Definition \ref{def:chebyshev} and Proposition \ref{prop:chebyshev}\,(iv).

For $j=1$ we use Definition \ref{def:chebyshev} and Proposition \ref{prop:chebyshev}\,(ii),(iv) to obtain
$$V_{q-3}(\lambda_q) = \lambda_q V_{q-2}(\lambda_q) - V_{q-1}(\lambda_q) = \lambda_q = V_1(\lambda_q).
$$ 
Now let $j\ge 2$. Using Definition \ref{def:chebyshev} and the induction hypothesis (for two previous values) we obtain 
\begin{eqnarray*}
V_{(q-2)-j}(\lambda_q) & = & \lambda_q V_{(q-2)-(j-1)}(\lambda_q) - V_{(q-2)-(j-2)}(\lambda_q) \\
& = & \lambda_q V_{j-1}(\lambda_q) - V_{j-2}(\lambda_q) \\
& = & V_j(\lambda_q)
\end{eqnarray*}
which proves the inductive step. 
\end{proof}

\section{Computing the entries of the frieze patterns} \label{sec:algorithm}

By definition, the entries of a frieze pattern have to satisfy the unimodular rule. However, the entries of 
a frieze pattern (with non-zero entries) 
satisfy many more conditions, the so-called {\em Ptolemy relations}. We briefly recall
this viewpoint here; see for instance \cite[Section 2]{HJ17} for more details.  
Let $F$ be a frieze pattern of height $n-3$, and denote the entries by $f_{i,j}$, as in Definition \ref{def:frieze}. 
Note that the indices appearing in the triangular fundamental region (highlighted in red in Definition \ref{def:frieze})
are in bijection with the diagonals of a regular $n$-gon (where the 1's at the bottom correspond to the 
edges of the polygon). Since the frieze pattern has a glide symmetry, this viewpoint implies that 
a frieze pattern of height $n-3$ can also be described as an assignment of (positive real) numbers
to the edges and diagonals of an $n$-gon $P$ such that edges are mapped to 1 and for each pair $(i,j)$ and $(k,\ell)$ of crossing diagonals, the Ptolemy relation 
$$f_{i,j}f_{k,\ell} = f_{i,k}f_{j,\ell} + f_{i,\ell}f_{k,j}
$$
holds (for a proof see for instance \cite[Theorem 3.3]{CHJ20}).
Note that the unimodular rule is just a special case of the Ptolemy relations, namely for the crossing diagonals
$(i,j)$ and $(i+1,j+1)$. 
\smallskip

Let $\mathcal{T}$ be a dissection of a polygon $P$. 
In Section \ref{sec:friezes} we have summarized the results from \cite{HJ17} showing how to obtain a frieze
pattern $F_{\mathcal{T}}$ corresponding to $\mathcal{T}$. More precisely, the quiddity row of $F_{\mathcal{T}}$
is specified (summing the labels of subpolygons at each vertex of $P$), and the results of \cite{HJ17} imply
that this is indeed the quiddity row of a frieze pattern (i.e. by repeatedly applying the unimodular rule 
one obtains a closing row of 1's).  

\smallskip

The aim of this section is to provide a fairly simple procedure, performed on the dissection $\mathcal{T}$, 
for computing all entries of the corresponding frieze pattern $F_{\mathcal{T}}$. That is,
we introduce a combinatorial algorithm which produces for every pair $(i,j)$ of vertices 
of the $n$-gon $P$ a (positive real) number $c_{i,j}$. We then go on to show that the numbers $c_{i,j}$
coincide with the entries $f_{i,j}$ of the frieze pattern $F_{\mathcal{T}}$ from \cite{HJ17}. An alternative algorithm,
using matchings, appears in \cite[Section 6]{BC21}.

\begin{Algorithm} \label{alg:frieze}
Let $P$ be a dissected polygon. For any vertex $i$ of $P$ perform the following steps:
\begin{enumerate}
\item[{(i)}] Put label 0 at vertex $i$. 
\item[{(ii)}] If $Q$ is a $q$-gon in the dissection which is adjacent to $i$, and $\ell$ is a vertex of $Q$, then put label
$V_{d-1}(\lambda_q)$ at vertex $\ell$, where $d$ is the number of edges in $Q$ from $i$ to $\ell$. 
\item[{(iii)}] Let $Q$ be a $q$-gon in the dissection for which labels $a$ and $b$ at the ends $s$ and $t$ of an edge
have already been assigned.
For every vertex $k$ of $Q$ put label 
$$aV_{q-d-2}(\lambda_q) + bV_{d-1}(\lambda_q)$$
at vertex $k$ where $d$ is the number of edges in $Q$ from $s$ to $k$ (on the path not going through $t$).
\end{enumerate}
Eventually, every vertex $j$ of $P$ gets assigned a label, denoted $c_{i,j}$. 
\end{Algorithm}

\begin{Remark}
Note that in Step (ii) of Algorithm \ref{alg:frieze} it doesn't matter whether the distance $d$ is taken clockwise
or anticlockwise because by Proposition \ref{prop:symmetry} we have 
$$V_{d-1}(\lambda_q) = V_{q-2-(d-1)}(\lambda_q) = V_{q-d-1}(\lambda_q).
$$
Similarly, in Step (iii) the roles of $a$ and $b$ can be interchanged since
$$b V_{q-(q-d-1)-2}(\lambda_q) + a V_{q-d-2}(\lambda_q) = b V_{d-1}(\lambda_q)+aV_{q-d-2}(\lambda_q).
$$
\end{Remark}

\begin{Theorem} \label{thm:BCIdissections}
Let $\mathcal{T}$ be a dissection of a regular $n$-gon $P$. 
Then for all vertices $i$ and $j$ of $P$,
the numbers $c_{i,j}$ computed with Algorithm \ref{alg:frieze} coincide with the entries in the 
frieze pattern $F_{\mathcal{T}}$ from \cite{HJ17}, that is, $c_{i,j}=f_{i,j}$.  
\end{Theorem}

\begin{proof}
We argue inductively via the number of subpolygons in the dissection $\mathcal{T}$. 

First, let $\mathcal{T}$ consist of a single $q$-gon, with vertices $1,2,\ldots,q$ in counterclockwise order. 
The entries in the frieze pattern $F_{\mathcal{T}}$ satisfy $f_{i,i+1}=1=V_0(\lambda_q)=c_{i,i+1}$ and 
$f_{i,i+2}=\lambda_q=V_1(\lambda_q)=c_{i,i+2}$
(these are the entries in the quiddity row), see Algorithm \ref{alg:frieze}\,(ii). Moreover, since $F_{\mathcal{T}}$ is a frieze pattern 
by \cite{HJ17}, the entries 
satisfy all Ptolemy relations. Start at some vertex, say at vertex 1. For the crossing diagonals $(1,3)$ and $(2,4)$
the Ptolemy relations yield
$$f_{1,4} = f_{1,3}f_{2,4} - 1 = \lambda_q V_1(\lambda_q) - V_0(\lambda_q) = V_2(\lambda_q).
$$
Inductively, assume that we have already shown that $f_{1,k}= V_{k-2}(\lambda_q)$ for all $2\le k\le j$. Then
we consider the crossing diagonals $(1,j)$ and $(j-1,j+1)$. The Ptolemy relations in this case, together with
the induction hypothesis, give
$$c_{1,j+1} = c_{j-1,j+1}c_{1,j} -c_{1,j-1} =  \lambda_q V_{j-2}(\lambda_q) - V_{j-3}(\lambda_q) = V_{j-1}(\lambda_q),
$$ 
as desired. This shows that in the case of a single $q$-gon we indeed have $f_{i,j}=c_{i,j}$ for all $1\le i,j\le q$,
see Step (ii) of Algorithm \ref{alg:frieze}.
\smallskip

Now we consider an arbitrary dissection $\mathcal{T}$ of some regular polygon $P$. Take a subpolygon 
$P_1$ in this dissection whose 
boundary contains only one diagonal from $P$ and all other edges of the subpolygon are also edges of $P$. 
(Clearly, such a subpolygon exists, in the case of triangulations this has sometimes be called an 'ear'.) 
Denote the number of vertices of $P_1$ by $q$. 
Let $P'$ be the polygon obtained from $P$ by deleting the $q$-gon $P_1$. Inductively, we can assume that
for every pair of vertices $(i,j)$ inside the smaller polygon $P'$ we have $f_{i,j}=c_{i,j}$. Moreover,
by the argument above for a single $q$-gon, we can assume that $f_{u,v}=c_{u,v}$ for all pairs of vertices
inside the $q$-gon $P_1$. 

Denote the 
two endpoints of the unique edge belonging to $P'$ and to $P_1$ by $s$ and $t$. Let $\ell_1,\ldots,
\ell_{q-2}$ the other vertices of the $q$-gon, as in the following figure.
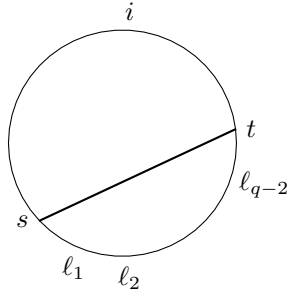
\begin{center}
\begin{tikzpicture}[auto]
    \node[name=s, draw, shape=regular polygon, regular polygon sides=50, minimum size=3cm] {};
    \draw[thick] (s.corner 20) to (s.corner 40);
   \draw[shift=(s.corner 1)] node[above] 
   {{\small $i$}};
  \draw[shift=(s.corner 20)]  node[left] 
  {{\small $s$}};
  \draw[shift=(s.corner 40)]  node[right]  
  {{\small $t$}};
  \draw[shift=(s.corner 23)]  node[below]  
  {{\small $\ell_1$}};
  \draw[shift=(s.corner 27)]  node[below] 
  {{\small $\ell_2$}};
  \draw[shift=(s.corner 36)]  node[right] 
  {{\small $\ell_{q-2}$}};
   \end{tikzpicture}
\end{center}
Let $i$ be a vertex of $P'$. 
In the large polygon $P$, we consider the crossing diagonals $(i,\ell_1)$ and $(s,t)$. The frieze pattern $F_{\mathcal{T}}$
satisfies all Ptolemy relations, so we have
$$f_{i,\ell_1}\cdot f_{s,t} = f_{i,s}\cdot f_{t,\ell_1} + f_{i,t}\cdot f_{s,\ell_1}.
$$
By induction and the case of a single $q$-gon, we know that $f_{s,t}=1$ and $f_{i,s}=c_{i,s}$ and $f_{i,t}=c_{i,t}$
and $f_{t,\ell_1}=V_1(\lambda_q)=\lambda_q$ and $f_{s,\ell_1}=V_0(\lambda_q)=1$.
Thus, together with Proposition \ref{prop:symmetry}, we get
$$f_{i,\ell_1} = c_{i,s}\cdot V_1(\lambda_q) + c_{i,t}\cdot V_0(\lambda_q)
= c_{i,s}\cdot V_{q-3}(\lambda_q) + c_{i,t}\cdot V_0(\lambda_q).
$$
According to Step (iii) of Algorithm \ref{alg:frieze} we can conclude that $f_{i,\ell_1}=c_{i,\ell_1}$. 

We then proceed inductively along the vertices of the $q$-gon $P_1$. Suppose that we have already shown that
$f_{i,\ell_k} = c_{i,\ell_k}$ for all $1\le k\le j$. Then we consider in the large polygon $P$ the crossing 
diagonals $(i,\ell_j)$ and $(\ell_{j-1},\ell_{j+1})$. The Ptolemy relations for the frieze pattern $F_{\mathcal{T}}$ 
yield
$$f_{i,\ell_{j+1}} = f_{\ell_{j-1},\ell_{j+1}}\cdot f_{i,\ell_j} - f_{i,\ell_{j-1}}
= \lambda_q c_{i,\ell_j} - c_{i,\ell_{j-1}}.
$$
By Step (iii) of Algorithm \ref{alg:frieze} this takes the following form.
\begin{eqnarray*}f_{i,\ell_{j+1}} & = & \lambda_q (c_{i,s}V_{q-j-2}(\lambda_q) + c_{i,t}V_{j-1}(\lambda_q))
- (c_{i,s}V_{q-j-1}(\lambda_q) + c_{i,t}V_{j-2}(\lambda_q)) \\
& = & c_{i,s}( \lambda_q V_{q-j-2}(\lambda_q) - V_{q-j-1}(\lambda_q)) 
+ c_{i,t} ( \lambda_q V_{j-1}(\lambda_q) - V_{j-2}(\lambda_q)) \\
& = & c_{i,s} V_{q-j-3}(\lambda_q) + c_{i,t}V_j(\lambda_q) \\
& = & c_{i,\ell_{j+1}}
\end{eqnarray*}
This shows that indeed for every pair of vertices $(i,j)$ of $P$ we have $f_{i,j}=c_{i,j}$, as claimed. 
\end{proof}

\begin{Corollary} \label{cor:ge1}
Let $\mathcal{T}$ be a dissection of a regular $n$-gon $P$, and let $p\ge 3$ be the smallest size of a polygon
appearing in $\mathcal{T}$. Then all entries in the frieze pattern $F_{\mathcal{T}}$ corresponding to $\mathcal{T}$ 
belong to the set $\{1\}\cup [\lambda_p,\infty)$. 
\end{Corollary}

\begin{proof} 
Let $f_{i,j}$ be an arbitrary entry of $F_{\mathcal{T}}$. Theorem \ref{thm:BCIdissections} tells us that 
$f_{i,j}$ can be computed inductively by Algorithm \ref{alg:frieze}. If $i$ and $j$ are in the same $q$-gon, 
then Remark \ref{rem:Euclidlength} implies that $f_{i,j}=1$ or $f_{i,j}\ge \lambda_q\ge \lambda_p$
(the last inequality holds by definition of $p$). If $i$ and $j$ are not in a same subpolygon, then 
inductively we can assume for the last application of 
Step (iii) of Algorithm \ref{alg:frieze} that $a\ge \lambda_p$ and $b\ge 1$ (or vice versa),
and then 
$$f_{i,j} = aV_{q-d-2}(\lambda_q) + bV_{d-1}(\lambda_q) \ge 
\lambda_p+1> \lambda_p$$
which proves the claim in the corollary. 
\end{proof}


\section{Frieze patterns of types $\Lambda_4$ and $\Lambda_6$}
\label{sec:lambda46}

Frieze patterns of type $\Lambda_p$ are defined in terms of their quiddity rows (see Definition 
\ref{def:lambdap}). However, it would be desirable to characterise frieze patterns in terms of properties of
all their entries, like in the case $p=3$ where the frieze patterns of type $\Lambda_3$ are 
precisely the frieze patterns with positive integral entries considered by Conway-Coxeter. 
It is an open question whether there is a similarly useful characterisation of frieze patterns of
type $\Lambda_p$ for $p\ge 4$, see \cite[Section 7]{HJ17}. As an application of the counting 
procedure for dissections in Algorithm \ref{alg:frieze} we present in this section such a characterisation 
for the cases $p=4$ and $p=6$. Frieze patterns of type $\Lambda_4$ and $\Lambda_6$ have 
also been considered by Andritsch \cite{A20}, but without providing a complete characterisation.
\medskip

We recall that $\lambda_4=2\cos(\frac{\pi}{4})= \sqrt{2}$ and 
$\lambda_6=2\cos(\frac{\pi}{6})= \sqrt{3}$. Hence frieze patterns of types $\Lambda_4$ and
$\lambda_6$ are defined by their quiddity rows consisting of positive integral multiplies of
$\sqrt{2}$ and $\sqrt{3}$, respectively. 

In a frieze pattern, we number the row of 1s as the first row, the quiddity row as second row and so on.

\begin{Theorem} \label{thm:lambda46}
Let $\mathcal{F}$ be a frieze pattern. Then the following statements are equivalent:
\begin{enumerate}
\item[{(a)}] $\mathcal{F}$ is a frieze pattern of type $\Lambda_4$ or $\Lambda_6$.
\item[{(b)}] The entries of $\mathcal{F}$ satisfy the following two conditions. 
\begin{itemize}
\item[{(i)}] The odd-numbered rows consist of positive integers.
\item[{(ii)}] The even-numbered rows consist of positive integral multiples of $\sqrt{2}$ (for type $\Lambda_4$)
or $\sqrt{3}$ (for type $\Lambda_6$), respectively.
\end{itemize}
\end{enumerate}
\end{Theorem}

\begin{proof}
First assume that a frieze pattern satisfies properties (i) and (ii). In particular, by (i), the quiddity row
consists of positive integral multiplies of $\sqrt{2}$ or $\sqrt{3}$. Then by Definition \ref{def:lambdap},
the frieze pattern is of type $\Lambda_4$ or $\Lambda_6$. 
\smallskip

Conversely, suppose that we have a frieze pattern $\mathcal{F}$ of type $\Lambda_4$ or $\Lambda_6$. 
By Theorem \ref{thm:HJ-ThmA} there exists a corresponding 4-angulation or 6-angulation of a polygon
such that the entries in the quiddity row are of the form $a_v\sqrt{2}$ or $a_v\sqrt{3}$ where $a_v$ is the 
number of quadrangles or hexagons attached at the vertex $v$ of the polygon. 
Theorem \ref{thm:BCIdissections} tells us that we can compute the entries of $\mathcal{F}$
by using Algorithm \ref{alg:frieze}. 

We now consider the cases $p=4$ and $p=6$ separately. 

Let $p=4$ and let $P$ be a polygon dissected into quadrangles. For any vertex $i$ the values $f_{i,j}$
in the frieze pattern $\mathcal{F}$ are inductively determined by Algorithm \ref{alg:frieze}. In any 
quadrangle attached to $i$ the neighbouring vertices get label $V_0(\sqrt{2})=1$, the remaining vertex 
gets label $V_1(\sqrt{2})=\sqrt{2}$. Note that the labels are positive integers or integral multiples 
of $\sqrt{2}$, depending on the distance to $i$ being odd or even (i.e. the corresponding entries in
$\mathcal{F}$ appearing in odd-numbered or even-numbered rows). When inductively a further 
quadrangle is glued and values $a$ and $b$ at one of the edges of the quadrangle are already known, we
can hence assume that $a$ is a positive integer and $b$ a multiple of $\sqrt{2}$, or vice versa. 
By symmetry, let us assume that $a$ is a positive integer. 
The two new labels appearing in the quadrangle are given by 
$a\sqrt{2}+b$ and $a+b\sqrt{2}$, thus the first one is an integral multiple of $\sqrt{2}$ and the other 
is a positive integer. By induction, this proves that the labels from positive integers and integral
multiples of $\sqrt{2}$ alternate with the distance from the starting vertex $i$. This means that 
 properties (i) and (ii) are satisfied for the frieze pattern $\mathcal{F}$. 
 \smallskip
 
Let $p=6$ and let $P$ be a polygon dissected into hexagons.
For any vertex $i$ the values $f_{i,j}$
in the frieze pattern $\mathcal{F}$ are inductively determined by Algorithm \ref{alg:frieze}. In any 
hexagon attached to $i$ the neighbouring vertices get label $V_0(\sqrt{2})=1$, the other vertices 
get labels $V_1(\sqrt{3})=\sqrt{3}$, $V_2(\sqrt{3})=\sqrt{3}^2-1= 2$, 
$V_3(\sqrt{3})= \sqrt{3}^3-2\sqrt{2}=\sqrt{3}$. Again, the labels are positive integers or integral multiples 
of $\sqrt{3}$, depending on the distance to $i$ being odd or even. When inductively a further 
hexagon is glued and values $a$ and $b$ at one of the edges of the quadrangle are already known, we
can (by symmetry) assume that $a$ is a positive integer and $b$ a multiple of $\sqrt{3}$. 
The new labels of the vertices in the glued hexagon are then given by 
$aV_3(\sqrt{3})+bV_0(\sqrt{3})=a\sqrt{3}+b$,
$aV_2(\sqrt{3})+bV_1(\sqrt{3})=2a+b\sqrt{3}$,
$aV_1(\sqrt{3})+bV_2(\sqrt{3})=a\sqrt{3}+2b$, and
$aV_0(\sqrt{3})+bV_3(\sqrt{3})=a+b\sqrt{3}$.
Note that these values are alternating between integral multiples of $\sqrt{3}$ and positive integers,
depending on the distance of the vertices to the starting vertex being odd or even. Inductively, 
we can deduce that properties (i) and (ii) are satisfied for the frieze pattern $\mathcal{F}$. 
\end{proof}

Theorem \ref{thm:lambda46} gives a partial answer to a question in \cite[Section 7, Question (i)]{HJ17}.
It remains open to find a similar characterisation of frieze patterns of type $\Lambda_p$
for $p\not\in \{3,4,6\}$.

\begin{Remark} The proof of Theorem \ref{thm:lambda46} actually gives more precise 
information about the integers appearing in the odd-numbered rows of a frieze pattern $\mathcal{F}$ of type
$\Lambda_4$ or $\Lambda_6$. 

If $\mathcal{F}$ is of type $\Lambda_4$ then the odd-numbered rows consists of odd positive integers. 

If $\mathcal{F}$ is of type $\Lambda_6$ then the row with odd number $2k+1$ consists of 
positive integers which are congruent to $(-1)^k$ modulo 3. 

We leave the straightforward verification of these facts to the reader.   
\end{Remark}


\section{Infinite frieze patterns of type $\Lambda_p$}
\label{sec:infiniteLambdap}

Infinite frieze patterns of positive integers have been introduced by Tschabold \cite{T15}, see
also \cite{BPT16}. We generalize this notion by introducing infinite frieze patterns of
type $\Lambda_p$ for any $p\ge 3$ (where the case $p=3$ yields the infinite frieze patterns of
positive integers).

\begin{Definition} \label{def:infinitefrieze}
An {\em infinite frieze pattern} $\mathcal{F}$ is an array $(f_{i,j})_{i,j\in \mathbb{Z}, j - i\ge 1}$
of infinitely many infinite horizontal rows of positive real numbers, with an
offset between neighbouring rows, and with a bottom row consisting of entries
$f_{i,i+1}=1$ for $i\in \mathbb{Z}$. Sometimes it is useful to include a further bottom row of entries $f_{i,i}=0$.
The other rows contain
positive real numbers such that each diamond of four neighbouring numbers\,\,\,\,
$\begin{matrix} & b & \\ a & & d \\ & c & \end{matrix}$\,\,\,
satisfies the unimodular rule $ad-bc=1$.

\end{Definition}

An infinite frieze pattern can be visualized as in the following figure.
$$\begin{array}{ccccccccccccc}
& & \iddots & & \iddots & & \iddots & & \iddots & & \iddots & & \\
\cdots & & f_{-2,2} & & f_{-1,3} & & f_{0,4} & & f_{1,5} & & f_{2,6} & & \cdots \\
& & & f_{-1,2} & & f_{0,3} & & f_{1,4} & & f_{2,5} & & & \\
\cdots & & f_{-1,1} & & f_{0,2} & & f_{1,3} & & f_{2,4} & & f_{3,5} &  & \cdots \\
  & 1 & & 1 & & 1 & & 1 & & 1 & & 1 & \\
 \end{array}
$$

Comparing this with Definition \ref{def:frieze} the crucial difference is that in an infinite frieze pattern
there is no bounding row of 1's at the top. 
We can count the rows of an infinite frieze pattern starting with row 1 being the row of 1's at the bottom.
The second row $(f_{i,i+2})_{i\in \mathbb{Z}}$ of an infinite frieze pattern is called the {\em quiddity row}. 
\smallskip

According to Definition \ref{def:infinitefrieze}, the entries in an infinite frieze pattern have to satisfy 
the unimodular rule for each diamond. However, the following observation shows that the entries satisfy
many more relations. This is a well known fact for finite frieze patterns and can easily be transferred 
to infinite frieze patterns. 

\begin{Proposition}[Ptolemy relations] \label{prop:Ptolemy}
Let $\mathcal{F}=(f_{i,j})$ be an infinite frieze pattern. Then for all $i< j < k< \ell$ we have
$$f_{i,k}f_{j,\ell} = f_{i,j} f_{k,\ell} + f_{i,\ell} f_{j,k}.
$$
\end{Proposition} 

\begin{proof}
This is the infinite version of the analogous result for finite frieze patterns (with coefficients)
in \cite[Thm. 3.3]{CHJ20}; in fact, the proof carries over verbatim. Note that the 
tameness assumption is satisfied since all entries of an infinite frieze pattern are non-zero by Definition
\ref{def:infinitefrieze}, so the proof of \cite[Prop. 2.6]{CHJ20} shows that all $3\times 3$-submatrices
of $\mathcal{F}$ have determinant 0. Then the proofs of the earlier results \cite[Prop. 2.10\,(1)]{CHJ20}
and \cite[Cor. 2.11\,(1)]{CHJ20} needed for proving \cite[Thm. 3.3]{CHJ20} also carry over directly
(where in the proof of \cite[Prop. 2.10\,(1)]{CHJ20} one can ignore the two equations where the glide
symmetry is used). 
\end{proof}

In an infinite frieze pattern, a special role is played by the entries which are equal to 1.

\begin{Lemma} \label{lem:fund}
Let $\mathcal{F}=(f_{i,j})$ be an infinite frieze pattern. If $f_{i,j}=1$ for some $i+2\le j$ then 
the following part of $\mathcal{F}$ is a fundamental domain of a Coxeter frieze pattern (as in 
Definition \ref{def:frieze}). 
$$\begin{array}{ccccccccc}
 & & & & 1 & & & & \\
 & & & f_{i,j-1} & & f_{i+1,j} & & & \\
  & & \iddots & & &  & \ddots & & \\
 & f_{i,i+2} & & & & & & f_{j-1,j} & \\
 1 & & 1 &  & \ldots  & & 1 &  & 1 \\
 \end{array}
 $$
\end{Lemma}

\begin{proof}
We have to show that each diamond in the resulting array of numbers satisfies the unimodular rule. 
This is clear for all diamonds inside the fundamental domain (and all its copies obtained by successively
applying glide symmetry), as these are diamonds in the infinite frieze pattern $\mathcal{F}$. So we only have 
to consider diamonds which have entries from neighbouring copies of the fundamental region. These appear 
as in the following figure,
$$\begin{array}{ccccccccc} 
f_{i,j} & & f_{i,i+1} & & & & & & \\
& \ddots & & \ddots & & & & & \\
& & &  & f_{i,i+\ell} & & & & \\
& &  & f_{i+\ell,j} & & f_{i,i+\ell+1} & & & \\
& &  & & f_{i+\ell+1,j} & & & & \\
& & & &  &\ddots & & \ddots & \\
& & & &  & & f_{j-1,j} & & f_{i,j} \\
\end{array}
$$ 
with $i<i+\ell<i+\ell+1<j$ for some $\ell$. Then Proposition \ref{prop:Ptolemy} and the assumption
$f_{i,j}=1$ give that 
$$f_{i+\ell,j} f_{i,i+\ell+1} - f_{i,i+\ell} f_{i+\ell+1,j} = f_{i,j} f_{i+\ell,i+\ell+1} =1,
$$
proving the unimodular rule. 
\end{proof}

\medskip

We can now define infinite frieze patterns of type $\Lambda_p$, analogous to Definition 
\ref{def:lambdap}. Such infinite frieze patterns have also been considered by Banaian and Chen
\cite{BC21}. 

\begin{Definition} \label{def:infinitelambdap}
Let $p\ge 3$ be an integer. An infinite frieze pattern $\mathcal{F}=(f_{i,j})$ is {\em of type $\Lambda_p$} if the 
quiddity row $(f_{i,i+2})_{i\in \mathbb{Z}}$
consists of (necessarily positive) integral multiples of $\lambda_p=2\cos(\frac{\pi}{p})$. 
\end{Definition}

\begin{Example} \label{ex:lambda4}
For any $p\ge 3$, the quiddity row $(2\lambda_p)_{i\in \mathbb{Z}}$
yields an infinite frieze pattern of type $\Lambda_p$. 
\smallskip

In fact, by Proposition \ref{prop:chebyshev}\,(i) we know that for every real number $x$, in the following array of
Chebychev polynomials
$$\begin{array}{ccccccccccccc}
& & \iddots & & \iddots & & \iddots & & \iddots & & \iddots & & \\
& & V_5(x) & & V_5(x) & & V_5(x) & & V_5(x) & & V_5(x) & & \\
& V_4(x)& & V_4(x) & & V_4(x) & & V_4(x) & & V_4(x) & & V_4(x) & \\
\cdots & & V_3(x) & & V_3(x) & & V_3(x) & & V_3(x) & & V_3(x) & & \cdots \\
& V_2(x) & & V_2(x) & & V_2(x) & & V_2(x) & & V_2(x) & & V_2(x) & \\
\cdots & & x & & x & & x & & x & & x &  & \cdots \\
  & 1 & & 1 & & 1 & & 1 & & 1 & & 1 & 
 \end{array}
$$
all diamonds satisfy the unimodular rule (this also appears in \cite[Lemma 1]{BC21}). 
However, not every choice of $x\in \mathbb{R}$ 
yields an infinite frieze pattern, as in an infinite frieze pattern all entries must be positive
(see Definition \ref{def:infinitefrieze}). So we need to show that for $x=2\lambda_p$ all entries
in the above array are indeed positive real numbers. In fact, by Remark \ref{rem:roots} we have
$$
V_n(2\lambda_p) =  \prod_{k=1}^n \left(2\lambda_p- 2\cos\left(\frac{\pi}{n+1}k\right)\right)
\,=\, 2^n \prod_{k=1}^n \left(\lambda_p- \cos\left(\frac{\pi}{n+1}k\right)\right) >0
$$
since $\lambda_p = 2\cos \frac{\pi}{p}\ge 1$ and $|\cos(\frac{\pi}{n+1}k)|<1$ (as $1\le k\le n$).
 \smallskip
 
\smallskip

For instance, for $p=4$ we get the following infinite frieze pattern of type $\Lambda_4$:
$$\begin{array}{ccccccccccccc}
& & \iddots & & \iddots & & \iddots & & \iddots & & \iddots & & \\
& 239 & & 239 & & 239 & & 239 & & 239 & &239 & \\
\cdots & & 70\sqrt{2} & & 70\sqrt{2} & & 70\sqrt{2}& & 70\sqrt{2} & & 70\sqrt{2} & & \cdots \\
& 41& & 41& & 41 & & 41 & & 41 & &41 & \\
\cdots & & 12\sqrt{2} & & 12\sqrt{2} & & 12\sqrt{2}& & 12\sqrt{2} & & 12\sqrt{2} & & \cdots \\
& 7& & 7& & 7 & & 7 & & 7 & &7 & \\
\cdots & & 2\sqrt{2} & & 2\sqrt{2} & & 2\sqrt{2} & & 2\sqrt{2} & & 2\sqrt{2} &  & \cdots \\
  & 1 & & 1 & & 1 & & 1 & & 1 & & 1 & 
 \end{array}
$$

\end{Example}

\section{A combinatorial model for infinite frieze patterns of type $\Lambda_p$}
\label{sec:combmodel}

The aim of this section is to provide a combinatorial model for infinite frieze patterns of type $\Lambda_p$
for any $p\ge 3$. 
\smallskip

\begin{Definition}
We consider the two lines in the plane $\mathbb{R}^2$ of all points with second coordinate 0 and 1,
respectively. We denote the lower line by $I$ and the upper line by $II$. 
On the lower line we mark the points with integer coefficients and denote them 
by $i^I$ for $i\in \mathbb{Z}$.  On the upper line we can have any subset $S$ of the set of points with integer
coefficients; to distinguish them from the points on the lower line we denote them by $j^{II}$ for $j\in S$. 
Note that $S$ could be the empty set. 

We call these two lines with the marked points $i^I$ on the lower line and some subset $S$ of marked points 
on the upper line an {\em infinite strip}. 
The lower line $I$ is called an {\em infinity-gon}. 

An {\em arc} of the infinite strip is connecting two non-neighbouring marked points on the lower line 
(called {\em peripheral arc} and denoted $(i^I,j^I)$) or a marked point on the 
lower line with a marked point on the upper line (called {\em bridging arc} and denoted $(i^I,j^{II})$). 
These arcs are called admissible. 
We do not allow arcs connecting two non-neighbouring marked points on the upper line. 
Geometrically, arcs are realized as paths modulo isotopy relative to the end points. 

We call $(i^I,(i+1)^I)$ for $i\in \mathbb{Z}$ an {\em edge} of the lower line, and similarly for 
neighbouring vertices of the upper line. 
\end{Definition}

\begin{Definition}
Let $p\ge 3$. 
A $p$-angulation of an infinite strip with vertex set $S$ on the upper line is a maximal collection of 
admissible arcs which are pairwise non-crossing and such that the resulting subpolygons are $p$-gons.   
\smallskip

A $p$-angulation $\mathcal{T}$ of an infinite strip is called {\em locally finite} if each vertex $i^I$ on the lower line 
is connected to only finitely many arcs of $\mathcal{T}$. 
\smallskip
 
We consider two $p$-angulations to be equal if one can be obtained from the other by a translation of
the vertex sets on the upper line (that is, the $p$-angulations are equal up to a renumbering of the upper 
vertex sets induced by a translation). 
\end{Definition}

\smallskip

\begin{Example} \label{ex:4ang}
We give two examples of locally finite 4-angulations of infinite strips. The first example has infinitely many
vertices on the upper line, that is, $S=\mathbb{Z}$.
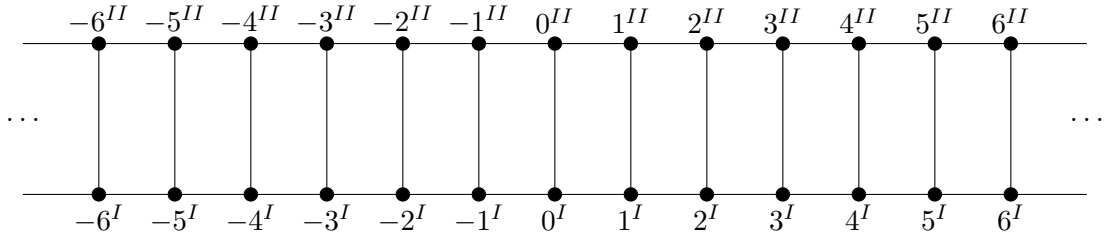
\begin{center}
\begin{tikzpicture}[scale=1]

      \draw[-] (-7,0) -- (7,0) node[right] {};
      \draw[-] (-7,2) -- (7,2) node[above] {};
       
       \node (N-7) at (-7,1) {$\ldots$};
      \node (C-60) at (-6,0) [circle, draw, inner sep=0pt, minimum width=5pt, fill=black]{};
      \node (N-6) at (-6,0) [below]{$-6^I$};
       \node (C-50) at (-5,0) [circle, draw, inner sep=0pt, minimum width=5pt, fill=black]{};
       \node (N-5) at (-5,0) [below]{$-5^I$};
       \node (C-40) at (-4,0) [circle, draw, inner sep=0pt, minimum width=5pt, fill=black]{};
       \node (N-4) at (-4,0) [below]{$-4^I$};
       \node (C-30) at (-3,0) [circle, draw, inner sep=0pt, minimum width=5pt, fill=black]{};
       \node (N-3) at (-3,0) [below]{$-3^I$};
       \node (C-20) at (-2,0) [circle, draw, inner sep=0pt, minimum width=5pt, fill=black]{};
       \node (N-2) at (-2,0) [below]{$-2^I$};
       \node (C-10) at (-1,0) [circle, draw, inner sep=0pt, minimum width=5pt, fill=black]{};
       \node (N-1) at (-1,0) [below]{$-1^I$};
      \node (C00) at (0,0) [circle, draw, inner sep=0pt, minimum width=5pt, fill=black]{};
      \node (N0) at (0,0) [below]{$0^I$};
    \node (C10) at (1,0) [circle, draw, inner sep=0pt, minimum width=5pt, fill=black]{};
    \node (N1) at (1,0) [below]{$1^I$};
    \node (C20) at (2,0) [circle, draw, inner sep=0pt, minimum width=5pt, fill=black]{};
    \node (N-2) at (2,0) [below]{$2^I$};
    \node (C30) at (3,0) [circle, draw, inner sep=0pt, minimum width=5pt, fill=black]{};
    \node (N3) at (3,0) [below]{$3^I$};
   \node (C40) at (4,0) [circle, draw, inner sep=0pt, minimum width=5pt, fill=black]{};
   \node (N4) at (4,0) [below]{$4^I$};
    \node (C50) at (5,0) [circle, draw, inner sep=0pt, minimum width=5pt, fill=black]{};
    \node (N5) at (5,0) [below]{$5^I$};
    \node (C60) at (6,0) [circle, draw, inner sep=0pt, minimum width=5pt, fill=black]{};
    \node (N6) at (6,0) [below]{$6^I$};
    \node (N7) at (7,1) {$\ldots$};
    
   \node (C-62) at (-6,2) [circle, draw, inner sep=0pt, minimum width=5pt, fill=black]{};
   \node (M-6) at (-6,2) [above]{$-6^{II}$};
       \node (C-52) at (-5,2) [circle, draw, inner sep=0pt, minimum width=5pt, fill=black]{};
        \node (M-5) at (-5,2) [above]{$-5^{II}$};
       \node (C-42) at (-4,2) [circle, draw, inner sep=0pt, minimum width=5pt, fill=black]{};
        \node (M-4) at (-4,2) [above]{$-4^{II}$};
       \node (C-32) at (-3,2) [circle, draw, inner sep=0pt, minimum width=5pt, fill=black]{};
        \node (M-3) at (-3,2) [above]{$-3^{II}$};
       \node (C-22) at (-2,2) [circle, draw, inner sep=0pt, minimum width=5pt, fill=black]{};
        \node (M-2) at (-2,2) [above]{$-2^{II}$};
       \node (C-12) at (-1,2) [circle, draw, inner sep=0pt, minimum width=5pt, fill=black]{};
        \node (M-1) at (-1,2) [above]{$-1^{II}$};
      \node (C02) at (0,2) [circle, draw, inner sep=0pt, minimum width=5pt, fill=black]{};
       \node (M0) at (0,2) [above]{$0^{II}$};
    \node (C12) at (1,2) [circle, draw, inner sep=0pt, minimum width=5pt, fill=black]{};
     \node (M1) at (1,2) [above]{$1^{II}$};
    \node (C22) at (2,2) [circle, draw, inner sep=0pt, minimum width=5pt, fill=black]{};
     \node (M2) at (2,2) [above]{$2^{II}$};
    \node (C32) at (3,2) [circle, draw, inner sep=0pt, minimum width=5pt, fill=black]{};
     \node (M3) at (3,2) [above]{$3^{II}$};
   \node (C42) at (4,2) [circle, draw, inner sep=0pt, minimum width=5pt, fill=black]{};
    \node (M4) at (4,2) [above]{$4^{II}$};
    \node (C52) at (5,2) [circle, draw, inner sep=0pt, minimum width=5pt, fill=black]{};
     \node (M5) at (5,2) [above]{$5^{II}$};
    \node (C62) at (6,2) [circle, draw, inner sep=0pt, minimum width=5pt, fill=black]{};
  \node (M6) at (6,2) [above]{$6^{II}$};
    
    \draw[-] (C-62) to (C-60);
    \draw[-] (C-52) to (C-50);
    \draw[-] (C-42) to (C-40);
    \draw[-] (C-32) to (C-30);
    \draw[-] (C-22) to (C-20);
    \draw[-] (C-12) to (C-10);
    \draw[-] (C02) to (C00);
    \draw[-] (C12) to (C10);
 \draw[-] (C22) to (C20);
    \draw[-] (C32) to (C30);
     \draw[-] (C42) to (C40);
    \draw[-] (C52) to (C50);
    \draw[-] (C62) to (C60);
    
\end{tikzpicture}
\end{center}

The second example has only one vertex on the upper line, that is $S=\{0^{II}\}$.
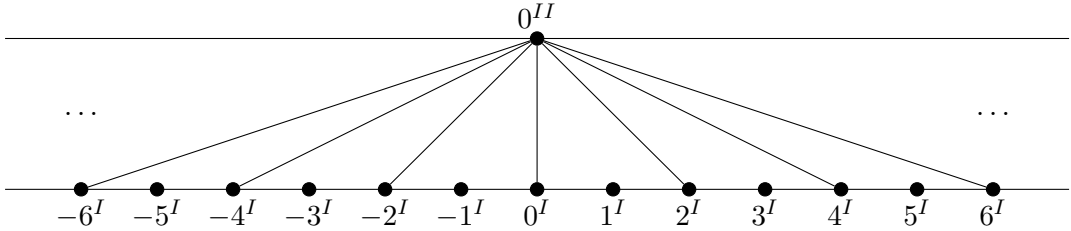
\begin{center}
\begin{tikzpicture}[scale=1]
      
      \draw[-] (-7,0) -- (7,0) node[right] {};
      \draw[-] (-7,2) -- (7,2) node[above] {};
       
       \node (D-6) at (-6,1) {$\ldots$};
        \node (C-60) at (-6,0) [circle, draw, inner sep=0pt, minimum width=5pt, fill=black]{};
        \node (N-6) at (-6,0) [below]{$-6^I$};
       \node (C-50) at (-5,0) [circle, draw, inner sep=0pt, minimum width=5pt, fill=black]{};
       \node (N-5) at (-5,0) [below]{$-5^I$};
       \node (C-40) at (-4,0) [circle, draw, inner sep=0pt, minimum width=5pt, fill=black]{};
       \node (N-4) at (-4,0) [below]{$-4^I$};
       \node (C-30) at (-3,0) [circle, draw, inner sep=0pt, minimum width=5pt, fill=black]{};
       \node (N-3) at (-3,0) [below]{$-3^I$};
       \node (C-20) at (-2,0) [circle, draw, inner sep=0pt, minimum width=5pt, fill=black]{};
       \node (N-2) at (-2,0) [below]{$-2^I$};
       \node (C-10) at (-1,0) [circle, draw, inner sep=0pt, minimum width=5pt, fill=black]{};
       \node (N-1) at (-1,0) [below]{$-1^I$};
      \node (C00) at (0,0) [circle, draw, inner sep=0pt, minimum width=5pt, fill=black]{};
      \node (N0) at (0,0) [below]{$0^I$};
    \node (C10) at (1,0) [circle, draw, inner sep=0pt, minimum width=5pt, fill=black]{};
    \node (N1) at (1,0) [below]{$1^I$};
    \node (C20) at (2,0) [circle, draw, inner sep=0pt, minimum width=5pt, fill=black]{};
    \node (N2) at (2,0) [below]{$2^I$};
    \node (C30) at (3,0) [circle, draw, inner sep=0pt, minimum width=5pt, fill=black]{};
    \node (N3) at (3,0) [below]{$3^I$};
   \node (C40) at (4,0) [circle, draw, inner sep=0pt, minimum width=5pt, fill=black]{};
   \node (N4) at (4,0) [below]{$4^I$};
    \node (C50) at (5,0) [circle, draw, inner sep=0pt, minimum width=5pt, fill=black]{};
    \node (N5) at (5,0) [below]{$5^I$};
    \node (C60) at (6,0) [circle, draw, inner sep=0pt, minimum width=5pt, fill=black]{};
    \node (N6) at (6,0) [below]{$6^I$};
    \node (D6) at (6,1) {$\ldots$};
    
      \node (C02) at (0,2) [circle, draw, inner sep=0pt, minimum width=5pt, fill=black]{};
      \node (M0) at (0,2) [above]{$0^{II}$};
 
    
    \draw[-] (C02) to (C-60);
    \draw[-] (C02) to (C-40);
    \draw[-] (C02) to (C-20);
    \draw[-] (C02) to (C00);
 \draw[-] (C02) to (C20);
     \draw[-] (C02) to (C40);
    \draw[-] (C02) to (C60);
    
\end{tikzpicture}
\end{center}
\end{Example}

\medskip

Our main result in this section, Theorem \ref{thm:bijection} below, 
will give a bijection between locally finite $p$-angulations of an infinite strip
and infinite frieze patterns of type $\Lambda_p$. Moreover, all entries of the frieze patterns can be computed 
from the combinatorial model using Algorithm \ref{alg:frieze} between vertices on the lower line. 
In particular, the quiddity row of the frieze pattern is given by the sequence of numbers of $p$-angles attached
to the vertices on the lower line. 
\smallskip

\begin{Example}
The first 4-angulation in Example \ref{ex:4ang} then corresponds to the infinite frieze pattern
of Example \ref{ex:lambda4} with quiddity row $(2\sqrt{2})_{i\in \mathbb{Z}}$,
and the second 4-angulation corresponds to the infinite frieze pattern
  $$\begin{array}{ccccccccccccc}
& & \iddots & & \iddots & & \iddots & & \iddots & & \iddots & & \\
\cdots & & ~4\sqrt{2}~ & & ~8\sqrt{2}~ & & ~4\sqrt{2}~ & & ~8\sqrt{2}~ & & ~4\sqrt{2}~ & & \cdots \\
& 7& & 7& & 7 & & 7 & & 7 & &7 & \\
\cdots & & 3\sqrt{2} & & 6\sqrt{2} & & 3\sqrt{2}& & 6\sqrt{2} & & 3\sqrt{2} & & \cdots \\
& 5& & 5& & 5 & & 5 & & 5 & &5 & \\
\cdots & & 2\sqrt{2} & & 4\sqrt{2} & & 2\sqrt{2}& & 4\sqrt{2} & & 2\sqrt{2} & & \cdots \\
& 3& & 3& & 3 & & 3 & & 3 & &3 & \\
\cdots & & \sqrt{2} & & 2\sqrt{2} & & \sqrt{2} & & 2\sqrt{2} & & \sqrt{2} &  & \cdots \\
  & 1 & & 1 & & 1 & & 1 & & 1 & & 1 & 
 \end{array}
$$
\end{Example}
\medskip

We now begin to proceed towards our desired main result, a bijection between locally finite 
$p$-angulations of infinite strips
with subsets $S\subseteq \mathbb{Z}$ on the upper line and infinite frieze patterns of type $\Lambda_p$. This 
bijection should be explicit in the sense that all entries $f_{i,j}$ of the infinite frieze pattern 
corresponding to a $p$-angulation of an infinite strip can be computed from the $p$-angulation
by using Algorithm \ref{alg:frieze}
between the vertices $i^I$ and $j^I$ on the lower line. 
\medskip

\begin{Remark} \label{rem:pangulationtofrieze}
We already know that for any locally finite $p$-angulation of an infinite strip we get an infinite frieze pattern of type
$\Lambda_p$ by using Algorithm \ref{alg:frieze} as described above. In fact, 
Theorem \ref{thm:BCIdissections} shows that the numbers obtained from Algorithm \ref{alg:frieze} 
are entries in a (finite) frieze pattern, so each diamond in the resulting array of these numbers satisfies 
the unimodular rule (note that even when starting with an infinite $p$-angulation, the numbers appearing 
in a fixed diamond are computed from Algorithm 4.1 only using a finite part of the locally finite $p$-angulation, so
the arguments of Theorem \ref{thm:BCIdissections} apply).  
\end{Remark} 
\medskip

The far less obvious part of our desired theorem is to show that every infinite frieze pattern of type 
$\Lambda_p$ can indeed be obtained from some locally finite $p$-angulation of an infinite strip with suitable
subset $S$ of vertices on the upper line. 
To this end, we will now start with an infinite frieze pattern of type $\Lambda_p$ and successively construct a corresponding
$p$-angulation of an infinite strip, that is, starting from the two lines $I$ and $II$ we provide a recipe for
which arcs to include, first the peripheral arcs then the bridging arcs.  
\smallskip

In Lemma \ref{lem:fund} we have seen that in an infinite frieze pattern the entries equal to 1 
play a special role. This leads to the following definition. 

\begin{Definition} \label{def:theta}
Let $\mathcal{F}=(f_{i,j})$ be an infinite frieze pattern of type $\Lambda_p$. Then we consider the following 
set of peripheral arcs of the infinite strip,
$$\Theta(\mathcal{F}) := \{ (i^I,j^I)\,|\,i+1<j,\,f_{i,j}=1\}.
$$ 
\end{Definition}

\begin{Lemma} \label{lem:angulation}
Let $\mathcal{F}=(f_{i,j})$ be an infinite frieze pattern of type $\Lambda_p$. For the corresponding set $\Theta(\mathcal{F})$
(cf. Definition \ref{def:theta}) the following holds. 
\begin{enumerate}
\item[{(a)}] If $(i^I,j^I)\in \Theta(\mathcal{F})$ then the arcs in $\Theta(\mathcal{F})$ below 
$(i^I,j^I)$ form a $p$-angulation of the finite polygon $P$ below $(i^I,j^I)$.
\item[{(b)}] $\Theta(\mathcal{F})$ is a set of pairwise non-crossing peripheral arcs. 
\item[{(c)}] Denoting by $\Theta(\mathcal{F})_P$ the $p$-angulation from part (a), the values
$f_{a,b}$ for $i\le a<b\le j$ agree with the values obtained by Algorithm \ref{alg:frieze} from 
$\Theta(\mathcal{F})_P$.  
\end{enumerate}
\end{Lemma}

\begin{proof}
\begin{enumerate}
\item[{(a)}] By Definition \ref{def:theta}, the assumption $(i^I,j^I)\in \Theta(\mathcal{F})$ means that
$f_{i,j}=1$. By Lemma \ref{lem:fund} the triangle below $f_{i,j}=1$ in $\mathcal{F}$ is the fundamental
domain of a finite frieze pattern of type $\Lambda_p$ (as introduced in \cite{HJ17}). We recall that
\cite[Theorem A]{HJ17} gives a bijection between finite frieze patterns of type $\Lambda_p$ of height $n$
and $p$-angulations of an $(n+3)$-gon. Under this bijection, the values 1 in the above fundamental domain 
correspond to the arcs in the $p$-angulation. 
\item[{(b)}] For a contradiction, suppose there are $i<j<k<\ell$ such that $(i^I,k^I)$ and $(j^I,\ell^I)$ are both in
$\Theta(\mathcal{F})$, that is, $f_{i,k}=1$ and $f_{j,\ell}=1$. The entry $f_{i,j}$ is in the triangle 
in $\mathcal{F}$ below $f_{i,k}=1$, and $f_{k,\ell}$ is in the triangle below $f_{j,\ell}=1$. By Lemma 
\ref{lem:fund} these triangles are fundamental domains of finite frieze patterns of type $\Lambda_p$.
In particular, all entries in these fundamental domains are $\ge 1$ (see Corollary \ref{cor:ge1}). 
Then using Proposition \ref{prop:Ptolemy} and the positivity of entries in an infinite 
frieze pattern (cf. Definition \ref{def:infinitefrieze}) gives the contradiction
$$1= f_{i,k}f_{j,\ell} = f_{i,j} f_{k,\ell} + f_{i,\ell}f_{j,k} > 1+0 = 1.
$$
\item[{(c)}] The values $f_{a,b}$ with $i\le a<b\le j$ are precisely the entries in the triangle of $\mathcal{F}$ below 
$f_{i,j}=1$, so they are the entries in a finite frieze pattern of type $\Lambda_p$ by the bijection
of \cite[Theorem A]{HJ17} mentioned above. By Theorem \ref{thm:BCIdissections} these values are 
computed using Algorithm \ref{alg:frieze}. 
\end{enumerate}
\end{proof}

\begin{Lemma} \label{lem:locallyfinite}
Let $\mathcal{F}$ be an infinite frieze pattern of type $\Lambda_p$. Then the set $\Theta(\mathcal{F})$
is locally finite (that is, each vertex $i^I$ is connected to only finitely many arcs from $\Theta(\mathcal{F})$). 
\end{Lemma}

\begin{proof}
Suppose that for some fixed $i\in \mathbb{Z}$ we had numbers $i<j<k<\ell<\ldots$ such that 
$(i^I,j^I),(i^I,k^I),(i^I,\ell^I),\ldots\in \Theta(\mathcal{F})$. The Ptolemy relation from 
Proposition \ref{prop:Ptolemy} for the numbers $i-1<i<j<k$ has the form
$$f_{i-1,j}f_{i,k} = f_{i-1,i}f_{j,k} + f_{i-1,k}f_{i,j}.
$$
Inserting that $f_{i-1,i}=1$ (by Definition \ref{def:infinitefrieze}) and $f_{i,j}=1=f_{i,k}$ by assumption 
we get
\begin{equation} \label{eq:1}
f_{i-1,k} = f_{i-1,j} - f_{j,k}.
\end{equation}
However, $f_{j,k}$ is an entry in the triangle below $f_{i,k}=1$, hence it is an entry in a finite 
frieze pattern of type $\Lambda_p$. Therefore $f_{j,k}\ge 1$ and Equation (\ref{eq:1}) gives
$$f_{i-1,k} \le f_{i-1,j}-1.
$$ 
Repeating the argument (with the values $i-1<i<k<\ell$) gives
$$f_{i-1,\ell} \le f_{i-1,k}-1 \le f_{i-1,j}-2,
$$ 
eventually contradicting the assumption that $\mathcal{F}$ has positive entries (by Definition \ref{def:infinitefrieze}). 
\smallskip

The case $(j^I,i^I),(k^I,i^I),(\ell^I,i^I),\ldots \in \Theta(\mathcal{F})$ is handled
symmetrically. 
\end{proof}

\begin{Lemma} \label{lem:longarcs}
Let $\mathcal{F}=(f_{i,j})$ be an infinite frieze pattern of type $\Lambda_p$ with corresponding set $\Theta(\mathcal{F})$. 
For $b_0\in\mathbb{Z}$ let $(b_{-1}^I,b_0^I)$ be the longest arc in $\Theta(\mathcal{F})$ going left from $b_0^I$
(if none exists, we set $b_{-1}^I=(b_0-1)^I$), and similarly let 
 $(b_{0}^I,b_1^I)$ be the longest arc in $\Theta(\mathcal{F})$ going right from $b_0^I$
(if none exists, we set $b_{1}^I=(b_0+1)^I$). Then we have
$$f_{b_0-1,b_0+1} = f_{b_{-1},b_1} + \lambda_p\cdot |\{\text{arcs in $\Theta(\mathcal{F})$ attached to $b_0^I$}\}|.
$$
\end{Lemma}

\begin{proof}
For abbreviation, let $i$ be the number of arcs in $\Theta(\mathcal{F})$ of the form $(b_0^I,x^I)$
(arcs going to the right of $b_0^I$), 
and let $k$ be the number of arcs in $\Theta(\mathcal{F})$ of the form $(x^I,b_0^I)$
(arcs going to the left of $b_0^I$). 

We first suppose that $b_{-1}<b_0-1$ and that $b_0+1<b_1$ (that is, there are arcs from $b_0^I$
going to the left and to the right). Below the arc $(b_0^I,b_1^I)\in \Theta(\mathcal{F})$ there is a 
$p$-angulation $\Theta(\mathcal{F})_P$ of a finite polygon $P$, see Lemma \ref{lem:angulation}\,(a).
In $P$, the vertices $b_1^I, b_0^I,(b_0+1)^I$ are consecutive, hence the value of the
finite frieze pattern corresponding to the arc $((b_0+1)^I, b_1^I)$ is $\lambda_p$ times the number of 
$p$-angles in $P$ attached to $b_0^I$, which is equal to
$$\lambda_p\cdot (1+ |\{\text{arcs in $\Theta(\mathcal{F})_P$ attached to $b_0^I$}\}|)
= \lambda_p\cdot i.
$$
By Lemma \ref{lem:angulation}\,(c), that frieze value is $f_{b_0+1,b_1}$, so we get
\begin{equation} \label{eq:fb1}
f_{b_0+1,b_1} = \lambda_p\cdot i.
\end{equation}
Symmetrically, we get
\begin{equation} \label{eq:fb-1}
f_{b_{-1},b_0-1} = \lambda_p\cdot k.
\end{equation}
Note that Equations (\ref{eq:fb1}) and (\ref{eq:fb-1}) remain true also when 
$b_1=b_0+1$ or $b_{-1}=b_0-1$ (since $i=0$ or $k=0$ in these cases).
Since $(b_0^I,b_1^I)\in \Theta(\mathcal{F})$, we get from the Ptolemy relation (see Proposition \ref{prop:Ptolemy}) and Equation (\ref{eq:fb1}) 
\begin{eqnarray*}
f_{b_0-1,b_0+1} & = & f_{b_0-1,b_0+1} f_{b_0,b_1} \,=\, f_{b_0-1,b_0} f_{b_0+1,b_1} + f_{b_0-1,b_1}f_{b_0,b_0+1} \\
& = & f_{b_0+1,b_1}+ f_{b_0-1,b_1}\,=\, \lambda_p\cdot i + f_{b_0-1,b_1}.
\end{eqnarray*}
Similarly, we also get
\begin{eqnarray*}
f_{b_0-1,b_1} & = & f_{b_0-1,b_1}f_{b_{-1},b_0} \,=\, f_{b_0-1,b_0}f_{b_{-1},b_1} +f_{b_{-1},b_0-1}f_{b_0,b_1} \\
& = & f_{b_{-1},b_1}+f_{b_{-1},b_0-1} \,=\, f_{b_{-1},b_1}+ \lambda_p\cdot k.
\end{eqnarray*}
Inserting the latter equation in the previous equation yields
$$f_{b_0-1,b_0+1} = f_{b_{-1},b_1}+ \lambda_p\cdot (i+k),
$$
which is the formula claimed in the Lemma. 
\end{proof}

\begin{Remark} \label{rem:belowarc}
We consider Lemma \ref{lem:longarcs} in the special case when the vertex $b_0^I$ is strictly below an arc
from $\Theta(\mathcal{F})$, that is, there exists $(i^I,j^I)\in \Theta(\mathcal{F})$ with $i<b_0<j$. 
\smallskip

By Lemma \ref{lem:angulation}, the part of $\mathcal{F}$ corresponding to the vertices below the arc $(i^I,j^I)$,
that is, the entries $f_{x,y}$ with $i\le x<y\le j$, form the fundamental domain of a finite frieze pattern of
type $\Lambda_p$. By \cite[Theorem A]{HJ17}, this corresponds to a $p$-angulation of the polygon $P$ 
below the arc $(i^I,j^I)$.
By definition, the vertices $b_{-1}^I$ and $b_1^I$ are contained in $P$ (as the longest arcs from $b_0^I$ must
not cross the arc $(i^I,j^I)$),
and the $p$-angulation contains a $p$-angle in which $b_{-1},b_0,b_1$ are consecutive vertices. 
Hence, $f_{b_{-1},b_1} =\lambda_p$ (see Algorithm \ref{alg:frieze}), so Lemma \ref{lem:longarcs}
gives
$$f_{b_0-1,b_0+1} = \lambda_p\cdot ( 1+ |\{ \text{arcs in $\Theta(\mathcal{F})$ attached to $b_0^I$}\}|).
$$ 
The right hand side is equal to $\lambda_p$ times the number of $p$-angles attached to $b_0^I$, so
the $p$-angulation of $P$ provides precisely the desired quiddity entry $f_{b_0-1,b_0+1}$. 
\end{Remark}

\begin{Remark} \label{rem:defect}
For any $b_0\in \mathbb{Z}$ we have
$$f_{b_0-1,b_0+1} - \lambda_p\cdot |\{ \text{arcs in $\Theta(\mathcal{F})$ attached to $b_0^I$}\}|
= f_{b_{-1},b_1} >0
$$
by Lemma \ref{lem:longarcs}. Since $\mathcal{F}$ is of type $\Lambda_p$, the left hand side is
in $\mathbb{Z}\cdot \lambda_p$, hence
$$f_{b_0-1,b_0+1} - \lambda_p\cdot |\{ \text{arcs in $\Theta(\mathcal{F})$ attached to $b_0^I$}\}|
\ge \lambda_p.
$$
We call the number
$$\mathrm{def}_{\mathcal{F}}(b_0^I) :=
f_{b_0-1,b_0+1} - \lambda_p\cdot ( 1+ |\{ \text{arcs in $\Theta(\mathcal{F})$ attached to $b_0^I$}\}|)
\in \mathbb{Z}_{\ge 0}\cdot \lambda_p
$$
the {\em defect} of the vertex $b_0^I$. 
\end{Remark}

\begin{Definition}
Let $\mathcal{F}$ be an infinite frieze pattern of type $\Lambda_p$ with corresponding set $\Theta(\mathcal{F})$.
A number $b_0\in \mathbb{Z}$ (and the vertex $b_0^I$) is called {\em saturated} for $\mathcal{F}$
if there is an arc $(i^I,j^I)\in \Theta(\mathcal{F})$
with $i<b_0<j$ (that is, the vertex $b_0^I$ is strictly below an arc from $\Theta(\mathcal{F})$). 
If there is no such arc, the number $b_0$ (and the vertex $b_0^I$) is called {\em non-saturated} for $\mathcal{F}$.   
\end{Definition}

\begin{Remark} \label{rem:infnonsat}
Let $\mathcal{F}$ be an infinite frieze pattern of type $\Lambda_p$. Then either all numbers (or vertices)
are saturated for $\mathcal{F}$ 
or there are infinitely many numbers (or vertices)  to both sides of the lower line $I$ which are 
non-saturated for $\mathcal {F}$.
\smallskip

If there are no non-saturated vertices then we are done. So let  
us suppose there were only finitely many non-saturated vertices, and let $y^I$ be the rightmost 
non-saturated vertex. Then $(y+1)^I$ is saturated, so there is an arc over $(y+1)^I$. Since $y^I$ is non-saturated,
this arc must start at $y^I$, that is, it is of the form $(y^I,y_1^I)\in \Theta(\mathcal{F})$ with $y<y_1$. 
Repeating the argument for the saturated vertex $y_1^I$ we find an arc $(y^I,y_2^I)\in \Theta(\mathcal{F})$ 
with $y<y_1<y_2$. Continuing in this way leads to infinitely many arcs in $\Theta(\mathcal{F})$
starting at $y^I$, contradicting the locally finiteness of $\Theta(\mathcal{F})$ shown in Lemma 
\ref{lem:locallyfinite}.

Symmetrically, there can also not be a leftmost non-saturated vertex. 
\end{Remark}

\begin{Lemma} \label{lem:consecutive}
Let $\mathcal{F}$ be an infinite frieze pattern of type $\Lambda_p$. If $x_1<x_2<\ldots <x_r$ are 
consecutive non-saturated numbers for $\mathcal{F}$ with 
$\mathrm{def}_{\mathcal{F}}(x_1^I)=\ldots=\mathrm{def}_{\mathcal{F}}(x_r^I)=0$, then
$r\le p-3$. 
\end{Lemma}

\begin{proof}
Suppose for a contradiction that $r= p-2$. Let $x_0$ be the non-saturated number preceding $x_1$ and
let $x_{r+1}$ be the non-saturated number succeeding $x_r$ (these numbers exist by Remark 
\ref{rem:infnonsat}). Since $x_0,x_1,\ldots,x_r,x_{r+1}$ are consecutive non-saturated numbers, we have
for every $i=0,1,\ldots,r$ 
that $(x_i^I,x_{i+1}^I)$ is an edge (that is, $x_{i+1}=x_i+1$) or 
the
arc $(x_i^I,x_{i+1}^I)$ must be in $\Theta(\mathcal{F})$. Moreover, if 
$(x_{i-1}^I,x_{i}^I)$ or $(x_i^I,x_{i+1}^I)$ are arcs, then they are are the longest arcs in $\Theta(\mathcal{F})$ going 
to the left and right from $x_i^I$, where $1\le i\le r$ (as otherwise, $x_{i-1}$ or $x_{i+1}$ would be 
saturated). By assumption and the definition of the defect in Remark \ref{rem:defect} 
we have
$$0 = \mathrm{def}_{\mathcal{F}}(x_i^I)=
f_{x_i-1,x_i+1} - \lambda_p\cdot ( 1+ |\{ \text{arcs in $\Theta(\mathcal{F})$ attached to $x_i^I$}\}|).
$$
On the other hand, from Lemma \ref{lem:longarcs} we get 
$$f_{x_i-1,x_i+1} = f_{x_{i-1},x_{i+1}} + \lambda_p\cdot |\{\text{arcs in $\Theta(\mathcal{F})$ attached to $b_0^I$}\}|.
$$
Combining these two equations yields
$$f_{x_{i-1},x_{i+1}} = \lambda_p\text{~~~for all $i=1,\ldots,r$}.
$$
Inductively using Ptolemy relations and Proposition \ref{prop:chebyshev}\,(i) it follows that 
the values $f_{x_0,x_2}$, $f_{x_0,x_3},\ldots$ are given by Chebyshev polynomials, more precisely we have
$$f_{x_0,x_j} = V_{j-1}(\lambda_p)\text{~~~for $j=1,2,\ldots, r+1=p-1$.}
$$
In particular, for $j=p-1$ we obtain from Proposition \ref{prop:chebyshev}\,(iv) that
$$ f_{x_0,x_{p-1}} = V_{p-2}(\lambda_p)=1,
$$
contradicting the assumption that $x_1,\ldots,x_r$ are non-saturated. 
\end{proof}

\begin{Theorem} \label{thm:pangulation}
Let $\mathcal{F}=(f_{i,j})$ be an infinite frieze pattern of type $\Lambda_p$. Then there is a locally finite
$p$-angulation of the infinite strip (with vertices $\mathbb{Z}$ on the lower line and 
vertices from a subset $S\subseteq \mathbb{Z}$ on the upper line) such that the entries $f_{i,j}$ are given
by performing Algorithm \ref{alg:frieze} between the vertices $i^I$ and $j^I$. In particular,
every quiddity entry $f_{i-1,i+1}$ is given by $\lambda_p$ times the number of $p$-angles attached to the vertex $i^I$. 
\end{Theorem}

\begin{proof}
We start by placing the arcs from the set $\Theta(\mathcal{F})$ on the lower line. Remark \ref{rem:defect}
implies that we can supplement by bridging arcs to the upper line in a way such that each vertex $b_0^I$ on the lower 
line has precisely $\frac{1}{\lambda_p}f_{b_0-1,b_0+1} -1$ arcs attached to it. Lemma \ref{lem:consecutive}
implies that this can be done in a way producing a $p$-angulation. Then the resulting $p$-angulation has the 
property that each vertex $b_0^I$ has precisely $f_{b_0-1,b_0+1}$ many $p$-angles attached to it. In particular,
the $p$-angulation is locally finite. 
The final statement then follows from Algorithm \ref{alg:frieze}.
\end{proof}

As the final step for our main result we now want to show that the $p$-angulation in Theorem \ref{thm:pangulation}
is unique (up to translation of the subset $S$ on the upper line). We start by considering the peripheral arcs.

\begin{Proposition} \label{prop:1inF}
Let $\mathcal{T}$ be a locally finite $p$-angulation of an infinite strip, with vertices $\mathbb{Z}$ on the lower line
$I$
and a subset $S\subseteq \mathbb{Z}$ of vertices on the upper line $II$. 
Let $\mathcal{F}=(f_{i,j})$ be the infinite frieze pattern of type $\Lambda_p$ obtained by performing 
Algorithm \ref{alg:frieze} with respect to $\mathcal{T}$ between the vertices on the lower line $I$. 
Then the following are equivalent for any numbers $i+1<j$:
\begin{enumerate}
\item[{(i)}] $f_{i,j}=1$,
\item[{(ii)}] $(i^I,j^I)$ is an arc in $\mathcal{T}$. 
\end{enumerate}
\end{Proposition}

\begin{proof}
The implication $(ii)\Rightarrow (i)$ is clear from Algorithm \ref{alg:frieze}. 

Conversely, suppose (for a contradiction) that $(i^I,j^I)$ is not an arc in the $p$-angulation $\mathcal{T}$.
We show that then $f_{i,j}>1$. There are two cases. 

We first consider the case where $i^I$ and $j^I$ belong to the same $p$-gon $P$ of $\mathcal{T}$. 
Note that this requires $p\ge 4$ (because in a triangulation any two vertices in a triangle are connected by an arc). 
Since $(i^I,j^I)$ is not an arc in $\mathcal{T}$, the vertices $i^I$ and $j^I$ are not consecutive vertices 
in $P$. Hence, $f_{i,j}=V_{\ell}(\lambda_p)$ for some $1\le \ell\le p-3$ (see Algorithm \ref{alg:frieze})
and by Remark 
\ref{rem:Euclidlength} we deduce $f_{i,j}\ge \lambda_p>1$ (for the final inequality we used that $p\ge 4$
in this case).

We now consider the case where $i^I$ and $j^I$ do not belong to a joint $p$-gon in $\mathcal{T}$. 
So when performing Algorithm \ref{alg:frieze} from $i^I$ to $j^I$ one goes through a sequence of at least 
two $p$-gons. In the final step, the Algorithm enters the last $p$-gon with numbers $a,b\ge 1$, and computes
$$f_{i,j} = a V_{p-d-2}(\lambda_p) + b V_{d-1}(\lambda_p) \ge 2>1
$$
for some distance $d\ge 1$. 
\end{proof}

As a summary of the above constructions and results we now obtain the main result of this section. 

\begin{Theorem} \label{thm:bijection}
Let $p\ge 3$. There is a bijection between infinite frieze patterns of type $\Lambda_p$ and
locally finite 
$p$-angulations of infinite strips (with vertices $\mathbb{Z}$ on the lower line and vertices
from a subset $S\subseteq \mathbb{Z}$ on the upper line), up to translational renumbering 
of the vertices on the upper line. All entries of the infinite frieze pattern corresponding to a
$p$-angulation can be computed with Algorithm \ref{alg:frieze}. 
\end{Theorem}

\begin{proof}
Let $\mathcal{F}$ be an infinite frieze pattern of type $\Lambda_p$. Proposition \ref{prop:1inF} says that 
if we want to obtain $\mathcal{F}$ from a $p$-angulation $\mathcal{T}$
of the infinite strip by using Algorithm \ref{alg:frieze}, then $\mathcal{T}$ must contain precisely the 
following peripheral arcs on the lower line: 
$$\{ (i^I,j^I)\,|\,i+1<j,\,f_{i,j}=1\} = \Theta(\mathcal{F}).
$$
Note that $\mathcal{T}$ must eventually satisfy 
\begin{equation} \label{eq:numberarcs}
\lambda_p\cdot (1+|\{\mbox{arcs in $\mathcal{T}$ attached to $i^I$}\}|) = f_{i-1,i+1}
\end{equation}
for all vertices $i^I$ on the lower line. 

If $i^I$ is strictly below an arc in $\Theta(\mathcal{F})$, then (\ref{eq:numberarcs}) is already satisfied
by arcs in $\Theta(\mathcal{F})$ by Remark \ref{rem:belowarc} (and below such arcs $\Theta(\mathcal{F})$
forms a $p$-angulation by Lemma \Ref{lem:angulation}\,(a).)

If $i^I$ is not strictly below an arc in $\Theta(\mathcal{F})$, then (\ref{eq:numberarcs}) tells us
precisely how many bridging arcs ending at $i^I$ we have to add to $\Theta(\mathcal{F})$ to
get $\mathcal{T}$. Note that the bridging arcs have to be added in a way such that we get $p$-angles;
this means that also the subset $S$ of vertices on the upper line is uniquely determined by $\mathcal{F}$,
up to translational renumbering of vertices on the upper line. 

This shows that the infinite frieze pattern $\mathcal{F}$ uniquely determines the $p$-angulation $\mathcal{T}$, 
up to a translational renumbering of the vertices on the upper line, thus proving the theorem.
\end{proof}

We conclude this section by stating two consequences of the results in this section. 
As a first consequence we obtain bounds on the entries 
of infinite frieze patterns of type $\Lambda_p$, analogous to Corollary \ref{cor:ge1}.

\begin{Corollary} \label{cor:infge1}
Let $\mathcal{F}$ be an infinite frieze pattern of type $\Lambda_p$. Then the entries of $\mathcal{F}$
belong to the set $\{1\}\cup [\lambda_p,\infty)$. 
\end{Corollary}

\begin{proof}
By Theorem \ref{thm:pangulation} there exists a locally finite $p$-angulation such that the entries 
$f_{i,j}$ of $\mathcal{F}$ are obtained by Algorithm \ref{alg:frieze}. Note that for any fixed entry
$f_{i,j}$, the computation of Algorithm \ref{alg:frieze} only uses a finite part of the $p$-angulation. 
So the Corollary follows from the finite version in Corollary \ref{cor:ge1}.  
\end{proof}

As another consequence we get an infinite version of Theorem \ref{thm:lambda46}.

\begin{Corollary} \label{cor:inflambda46}
Let $\mathcal{F}$ be an infinite frieze pattern. Then the following statements are equivalent:
\begin{enumerate}
\item[{(a)}] $\mathcal{F}$ is an infinite frieze pattern of type $\Lambda_4$ or $\Lambda_6$.
\item[{(b)}] The entries of $\mathcal{F}$ satisfy the following two conditions. 
\begin{itemize}
\item[{(i)}] The odd-numbered rows consist of positive integers.
\item[{(ii)}] The even-numbered rows consist of positive integral multiples of $\sqrt{2}$ (for type $\Lambda_4$)
or $\sqrt{3}$ (for type $\Lambda_6$), respectively.
\end{itemize}
\end{enumerate}
\end{Corollary}

\begin{proof} It is clear from Definition \ref{def:infinitelambdap} that (b) implies (a). Conversely, 
let $\mathcal{F}$ be an infinite frieze pattern of type $\Lambda_4$ or $\Lambda_6$.
By Theorem \ref{thm:pangulation} we know that there is a locally finite $4$- or $6$-angulation of an infinite strip
such that the entries of $\mathcal{F}$ can be computed by performing Algorithm \ref{alg:frieze}
on the vertices on the lower line. But then each value is computed by only using a finite part 
of the $4$- or $6$-angulation. So assertion (b) follows verbatim as in the proof of Theorem \ref{thm:lambda46}
for finite frieze patterns.  
\end{proof}

\section{Reflection groupoids of rank two and frieze patterns of type $\Lambda_p$}
\label{sec:groupoids}

The aim of this section is to present how 
infinite frieze patterns
appear naturally in the theory of Weyl groupoids
(of rank 2)
and to generalize the existing theory
to infinite frieze patterns of type $\Lambda_p$. 
On the side of the Weyl groupoids this leads to 
the new notion of a reflection groupoid of
type $\Lambda_p$.

\subsection{Reflection groupoids of type $\Lambda_p$}
The reflection groupoids that we introduce here generalize the Weyl groupoid in such a way that the reflections of a dihedral group are allowed. As a consequence we obtain the dihedral group as an additional example, and we get a natural appearance of infinite frieze patterns of type $\Lambda_p$.

Let $3\le p\in\mathbb{Z}$ and $\lambda_p=2\cos(\pi/p)$.
We write $\Ri$ for the ring of integers of 
 the algebraic number field $\mathbb{Q}(\lambda_p)$. Note that $\Ri\subseteq \mathbb{R}$. Moreover we write
$\mathbb{Z}_{\ge 0} \lambda_p$ and
$\mathbb{Z}_{\le 0} \lambda_p$
for the 
non-negative and non-positive
integral
multiples of $\lambda_p$.

We recall the notion of a Weyl groupoid \cite{p-CH09a} and generalize it to reflection groupoids whose Cartan entries lie in 
$\mathbb{Z}_{\le 0}\lambda_p$.
We need a slightly more general notion of generalized Cartan matrix which coincides with \cite[Section 1.1]{b-Kac90} for $p=3$:

\begin{Definition}[cf.\ {\cite[Section 1.1]{b-Kac90}}]
\label{def:genCartan}
Set $I=\{1,\ldots,r\}$. A \emph{generalized Cartan matrix of type $\Lambda_p$}
is a matrix $\Cm =(\cm _{ij})_{i,j\in I}$
in $\Ri^{I\times I}$ such that
\begin{enumerate}
\item[(M1)] For all $i,j,k\in I\:\: :\:\: \cm _{ii}=2$ and $\cm _{jk}\in 
\mathbb{Z}_{\le 0}\lambda_p$ if $j\not=k$.
\item[(M2)] For all $i,j\in I\:\::\:\: (\cm _{ij}=0 \:\:\Rightarrow\:\: \cm _{ji}=0)$.
\end{enumerate}
\end{Definition}

Like a Weyl group is determined by its Cartan matrix, a reflection groupoid is determined by a Cartan graph. A Cartan graph is a graph whose vertices (the set $A$ in the next definition) are labeled by Cartan matrices and whose edges are labeled by elements of $I$ where
a vertex $a\in A$ is connected with the vertex $\rho_i(a)$ by an edge labeled $i$:

\begin{Definition}[cf.\ {\cite[Def.\ 2.1]{p-CH09a}}]\label{def:cartangraph}
Set $I=\{1,\ldots,r\}$.
Let $A$ be a non-empty set, $\rfl _i : A \to A$ maps for all $i\in I$,
and $\Cm ^a=(\cm ^a_{jk})_{j,k \in I}$ a generalized Cartan matrix of type $\Lambda_p$
for all $a\in A$. The quadruple
\[ \Cc = \Cc (I,A,(\rfl _i)_{i \in I}, (\Cm ^a)_{a \in A})\]
is called a \emph{Cartan graph of type $\Lambda_p$} if
\begin{enumerate}
\item[(C1)] $\rfl _i^2 = \id$ for all $i \in I$,
\item[(C2)] $\cm ^a_{ij} = \cm ^{\rfl _i(a)}_{ij}$ for all $a\in A$ and $i,j\in I$.
\end{enumerate}
\end{Definition}

\begin{Remark}
Note that Definition \ref{def:cartangraph} is a special case of \cite[Definition 2.1]{p-CH11} when $r=2$.

Moreover, by Definition 
\ref{def:cartangraph},
every Cartan graph is
$|I|$-regular (i.e. each vertex has 
$|I|$ edges attached to it). 
\end{Remark}

\begin{Example}\label{ex:13122}
Let $r=2$, $A=\{a_1,\ldots,a_5\}$, 
$p=3$ and
$$\rfl_1(a_1)=a_2,\:
\rfl_1(a_3)=a_4,\:
\rfl_1(a_5) = a_5, \quad
\rfl_2(a_1)=a_1,\:
\rfl_2(a_2)=a_3,\:
\rfl_2(a_4)=a_5.$$
(Note that this determines involutions  
$\rho_i$, cf.\ Definition \ref{def:cartangraph}\,(C1).)
Setting
$$
C^{a_1} = \begin{pmatrix} 2 & -1 \\ -3 & 2 \end{pmatrix}, \quad
C^{a_2} = \begin{pmatrix} 2 & -1 \\ -2 & 2 \end{pmatrix}, \quad
C^{a_3} = \begin{pmatrix} 2 & -2 \\ -2 & 2 \end{pmatrix}, $$
$$
C^{a_4} = \begin{pmatrix} 2 & -2 \\ -1 & 2 \end{pmatrix}, \quad
C^{a_5} = \begin{pmatrix} 2 & -3 \\ -1 & 2 \end{pmatrix}, $$
we obtain a Cartan graph of type $\Lambda_3$. It is convenient to display it in the following form:
$$
\sigma_2 \circlearrowright
\begin{pmatrix} 2 & -1 \\ -3 & 2 \end{pmatrix}
\stackrel{\sigma_1}{\longleftrightarrow}
\begin{pmatrix} 2 & -1 \\ -2 & 2 \end{pmatrix}
\stackrel{\sigma_2}{\longleftrightarrow}
\begin{pmatrix} 2 & -2 \\ -2 & 2 \end{pmatrix}
\stackrel{\sigma_1}{\longleftrightarrow}
\begin{pmatrix} 2 & -2 \\ -1 & 2 \end{pmatrix}
\stackrel{\sigma_2}{\longleftrightarrow}
\begin{pmatrix} 2 & -3 \\ -1 & 2 \end{pmatrix}
\circlearrowleft \sigma_1
$$
where we label the edges with the reflections $\sigma_i$ defined in Definition \ref{def:WG} instead of $i$ in order to emphasize the groupoid structure.
\end{Example}

\begin{Example} \label{ex:2002}
Let $r=2$ and $p\ge 3$.
On the set $A=\mathbb{Z}$ we consider
involutions $\rho_1\neq \rho_2$ interchanging
neighbouring numbers.
As Cartan matrices we choose
$C^a= \begin{pmatrix} 2 & 0 \\ 0 & 2
\end{pmatrix}$ for all $a\in \mathbb{Z}$. 
The resulting Cartan graph of type $\Lambda_p$ can 
be visualized in the form
$$
\ldots \,
\stackrel{\sigma_1}{\longleftrightarrow}
\begin{pmatrix} 2 & 0 \\ 0 & 2 \end{pmatrix}
\stackrel{\sigma_2}{\longleftrightarrow}
\begin{pmatrix} 2 & 0 \\ 0 & 2 \end{pmatrix}
\stackrel{\sigma_1}{\longleftrightarrow}
\begin{pmatrix} 2 & 0 \\ 0 & 2 \end{pmatrix}
\stackrel{\sigma_2}{\longleftrightarrow}
\begin{pmatrix} 2 & 0 \\ 0 & 2 \end{pmatrix}
\stackrel{\sigma_1}{\longleftrightarrow}
\begin{pmatrix} 2 & 0 \\ 0 & 2 \end{pmatrix}
\stackrel{\sigma_2}{\longleftrightarrow}\, \ldots
$$
\end{Example}

\begin{Example} \label{ex:root2}
    Let $r=2$ and $p=4$, thus
    $\lambda_p=\lambda_4 = \sqrt{2}$.
On the set $A=\mathbb{Z}$ we consider
the same involutions $\rho_1$ and $\rho_2$ as in Example \ref{ex:2002}. As generalized 
Cartan matrices we now choose
$C^a= \begin{pmatrix} 2 & -2\sqrt{2} \\ 
-\sqrt{2} & 2
\end{pmatrix}$ for all $a$.
Then condition (C2) of Definition 
\ref{def:cartangraph} is satisfied. 
The resulting Cartan graph of type $\Lambda_4$ 
has the form
$$
\ldots 
\begin{pmatrix} 2 & -2\sqrt{2} \\ 
-\sqrt{2} & 2 \end{pmatrix}
\stackrel{\sigma_1}{\longleftrightarrow}
\begin{pmatrix} 2 & -2\sqrt{2} \\ 
-\sqrt{2} & 2 \end{pmatrix}
\stackrel{\sigma_2}{\longleftrightarrow}
\begin{pmatrix} 2 & -2\sqrt{2} \\ 
-\sqrt{2} & 2 \end{pmatrix}
\stackrel{\sigma_1}{\longleftrightarrow}
\begin{pmatrix} 2 & -2\sqrt{2} \\ 
-\sqrt{2} & 2 \end{pmatrix} \ldots
$$
\end{Example}

\medskip

We are now ready for the definition of a reflection groupoid. Note that there are several definitions of reflection groupoids in the literature, e.g.\ \cite[Def.\ 2.3]{p-CH11} (only rank two and only finitely many objects), \cite[Def.\ 2.10]{p-C21} (for Cartan entries in $\mathbb{R}$ but only the simply connected case).

\begin{Definition}[cf.\ {\cite[Def.\ 2.1]{p-CH09a}}]\label{def:WG}
Set $I=\{1,\ldots,r\}$.
Let $\{\alpha_i\mid i\in I\}$ be the set
of unit vectors of $\Ri^I$.
Let $\Cc = \Cc (I,A,(\rfl _i)_{i \in I}, (\Cm ^a)_{a \in A})$ be a
Cartan graph of type $\Lambda_p$. For all $i \in I$ and $a \in A$ define linear maps
$\s_i^a :\Ri^I\to \Ri^I$ by
$$
\s _i^a (\alpha_j) = \alpha_j - \cm _{ij}^a \alpha_i \qquad
\text{for all } j \in I.
\label{fwg_eq:sia}
$$
Note that the square of each $\sigma_i^a$ equals the identity.

\medskip

The \emph{reflection groupoid of} $\Cc $
is the category $\Wg (\Cc )$ such that $\Ob (\Wg (\Cc ))=A$ and
the morphisms are compositions of maps
$\s_i^a$ with $i\in I$ and $a\in A$,
where by abuse of notation $\s_i^a$ is considered as an element in $\Hom (a,\rfl_i(a))$ (the set of morphisms 
in the category $\Wg (\Cc )$ from the object
$a$ to the object $\rfl_i(a)$).

The cardinality of $I$ is the \emph{rank of}
the reflection groupoid $\Wg (\Cc )$.
If $p=3$, then a reflection groupoid is also called a \emph{Weyl groupoid}.
\end{Definition}

\begin{Remark}
For simplicity we will sometimes omit the object $a$ in $\s_i^a$ and write $\s_i$ if $a$ is given by the context. For example, we also write $\id ^a \s _i$
 for the morphism labeled $i$ with destination $a$.
\end{Remark}

\begin{Definition}[cf.\ {\cite[Def.\ 2.1]{p-CH09a}}]
\label{def:simplyconnected}
A Cartan graph is called \emph{standard} if $C^a = C^b$ for all $a,b \in A$.
In this case, if $A$ has only one object, then the groupoid is a group generated by reflections on $\Ri^I$, i.e.\ a reflection group.

A Cartan graph is called \emph{connected} if its reflection groupoid is connected, i.e.\ if the set of morphisms $\Hom (a,b)$ is non-empty for all objects $a,b\in A$ of $\Wg(\Cc)$. (Note that this means that
the Cartan graph is connected in the 
usual graph-theoretic meaning, i.e.\ between any pair of vertices there is 
a path.)

A Cartan graph is called \emph{simply connected} if $\Hom (a,a)=\{\id ^a\}$ for all $a\in A$, i.e.\ all morphisms corresponding to closed 
paths in the Cartan graph are equal to the identity.
\end{Definition}

\begin{Example}
For instance, the Cartan graph from Example \ref{ex:13122} is not simply connected because of the loops at the ends of the graph. The Cartan graphs in Examples \ref{ex:2002} and \ref{ex:root2} are simply connected as any closed path contains consecutive $\sigma_i^2$ which successively cancel as $\sigma_i^2$ is the identity. 
\end{Example}

In order to identify Cartan graphs that have the same structure, we later need the notion of equivalence:
\begin{Definition} 
\label{def:Cartanequiv}
Let $I=\{1,\ldots,r\}$.
Two Cartan graphs
$$\Cc = \Cc (I,A,(\rfl _i)_{i \in I}, (\Cm ^a)_{a \in A}), \quad \Cc' = \Cc (I,A',(\rfl' _i)_{i \in I}, ({\Cm'} ^a)_{a \in A'})$$
of type $\Lambda_p$ 
are called \emph{equivalent} if there are bijections $\varphi_0: I \to I$
and 
$\varphi_1: A \to A'$ 
such that
$$ \varphi_1(\rfl_i(a)) = \rfl'_{\varphi_0(i)}(\varphi_1(a))
\quad\text{and}\quad {c'}_{\varphi_0(i)\varphi_0(j)}^{\varphi_1(a)} = c_{ij}^a
$$
for all $i,j \in I$, $a \in A$.
\end{Definition}

\begin{Example}[An affine Weyl group]\label{affine41}
Let $I=\{1,2\}$, $A=\{a\}$, $\rfl_1(a)=a=\rfl_2(a)$ and consider the affine Cartan matrix of type $A_2^{(2)}$ 
(see \cite[Paragraph 4.8]{b-Kac90}):
$$ C^a :=
\begin{pmatrix} 2 & -4 \\ -1 & 2 \end{pmatrix}
. $$
Then $(I,A,(\rfl _i)_{i\in I},(\Cm ^a)_{a\in A})$ is a standard Cartan graph (of type $\Lambda_3$); the automorphism group $\Hom(a,a)$ is the affine Weyl group corresponding to $C^a$.
\end{Example}

\begin{Example}[The dihedral group]\label{dihedralgroup}
Let $p\ge 3$. 
Again, let $I=\{1,2\}$, $A=\{a\}$, $\rfl_1(a)=a=\rfl_2(a)$. This time we consider the matrix:
$$ C^a :=
\begin{pmatrix} 2 & -\lambda_p \\ -\lambda_p & 2 \end{pmatrix}
. $$
Then $(I,A,(\rfl _i)_{i\in I},(\Cm ^a)_{a\in A})$ is a standard Cartan graph of type $\Lambda_p$. 

The reflections as defined in Definition \ref{def:WG}
are given by 
\begin{equation} \label{eq:sigma12}
\sigma_1^a\begin{pmatrix} 1\\0 \end{pmatrix}
= -\begin{pmatrix} 1\\0 \end{pmatrix},~~
\sigma_1^a\begin{pmatrix} 0\\1 \end{pmatrix}
= \begin{pmatrix} \lambda_p\\1 \end{pmatrix}
\mbox{~~~and~~~}
\sigma_2^a\begin{pmatrix} 1\\0 \end{pmatrix}
= \begin{pmatrix} 1\\ \lambda_p \end{pmatrix},~~
\sigma_2^a\begin{pmatrix} 0\\1 \end{pmatrix}
= -\begin{pmatrix} 0\\1 \end{pmatrix}.
\end{equation}
Recall the Chebyshev polynomials from 
Definition \ref{def:chebyshev}: 
$V_{-1}(x)=0$, $V_{0}(x)=1$,
and $V_{n+1}(x)=xV_n(x)-V_{n-1}(x)$. 
From (\ref{eq:sigma12}) we inductively get
\begin{equation} \label{eq:sigma12r}
(\sigma_1\sigma_2)^r \begin{pmatrix} 1\\0 
\end{pmatrix} = \begin{pmatrix} V_{2r}(\lambda_p)
\\ V_{2r-1}(\lambda_p) \end{pmatrix}
\mbox{~~~and~~~}
(\sigma_1\sigma_2)^r \begin{pmatrix} 0\\1 
\end{pmatrix} = - \begin{pmatrix} V_{2r-1}(\lambda_p)
\\ V_{2r-2}(\lambda_p) \end{pmatrix}
\end{equation}

\begin{equation} \label{eq:sigma212r}
\sigma_2(\sigma_1\sigma_2)^r \begin{pmatrix} 1\\0 
\end{pmatrix} = \begin{pmatrix} V_{2r}(\lambda_p)
\\ V_{2r+1}(\lambda_p) \end{pmatrix}
\mbox{~~~and~~~}
\sigma_2(\sigma_1\sigma_2)^r \begin{pmatrix} 0\\1 
\end{pmatrix} = - \begin{pmatrix} V_{2r-1}(\lambda_p)
\\ V_{2r}(\lambda_p) \end{pmatrix}
\end{equation}
and symmetric formulae for $(\sigma_2\sigma_1)^r$
and $\sigma_1(\sigma_2\sigma_1)^r$.
From Proposition \ref{prop:chebyshev} we know that 
$V_{p-1}(\lambda_p) = 0 = -V_{-1}(\lambda_p)$ and
$V_p(\lambda_p)= -1 = -V_0(\lambda_p)$. From this we
can deduce inductively that $V_{2p-2}(\lambda_p)=
-V_{p-2}(\lambda_p) =-1$, $V_{2p-1}(\lambda_p)=
-V_{p-1}(\lambda_p) =0$ and 
$V_{2p}(\lambda_p)=
-V_{p}(\lambda_p) =1$. Inserting this into 
(\ref{eq:sigma12r}) shows that 
$(\sigma_1\sigma_2)^p$ is the identity map. 
Thus the automorphism group 
$\Hom(a,a)$ is isomorphic to the dihedral group of type $I_2(p)$, the symmetry group of the regular $p$-gon. 
\end{Example}

\begin{Definition}\label{def:realroots}
Let $\Cc $ be a Cartan graph of type $\Lambda_p$. 
For all $a\in A$ let
\[ \rer a=\{ \id ^a \s _{i_1}\cdots \s_{i_k}(\alpha_j) \mid
k\in \mathbb{N} _0,\:\: i_1,\dots,i_k,j\in I\}\subseteq \Ri^I.\]
The elements of the set $\rer a$ are called \emph{real roots} (at $a$).
The pair $(\Cc ,(\rer a)_{a\in A})$ is denoted by $\rsC \re (\Cc )$.
A real root $\alpha\in \rer a$ for $a\in A$ is called \emph{positive}
(resp.\ \emph{negative}) if $\alpha\in \mathbb{R}_{\ge 0}^I$ (resp.\ $\alpha\in -\mathbb{R}_{\ge 0}^I$).
\end{Definition}

\begin{Example}
We determine the real roots for the Cartan graph 
of type $\Lambda_p$ 
in Example \ref{dihedralgroup} which is related to the 
dihedral group. According to (\ref{eq:sigma12r})
and (\ref{eq:sigma212r}) and the fact deduced above
that 
$(\sigma_1\sigma_2)^p$ and $(\sigma_2\sigma_1)^p$ are 
the identity, the set of real roots is given by 
$$(R^{\mathrm{re}})^a = \left\{ 
\pm \begin{pmatrix} V_{2r}(\lambda_p)
\\ V_{2r-1}(\lambda_p) \end{pmatrix} \,|\,
0\le r\le p-1\right\}
\cup 
\left\{ 
\pm \begin{pmatrix} V_{2r}(\lambda_p)
\\ V_{2r+1}(\lambda_p) \end{pmatrix}\,|\,
0\le r\le p-1\right\}
$$
(and the positive roots are those with a $+$ sign).
\end{Example}

\begin{Example}\label{ex:realroots}
Consider the Cartan graph from Example \ref{ex:13122}. The positive real roots are:
\begin{eqnarray*}
\rer {a_1}_+ &=& \left\{\text{columns of }
\begin{pmatrix}
1 & 1 & 1 & 1 & 0 \\
0 & 1 & 2 & 3 & 1
\end{pmatrix}\right\} \\
\rer {a_2}_+ &=& \left\{\text{columns of }
\begin{pmatrix}
1 & 2 & 3 & 1 & 0 \\
0 & 1 & 2 & 1 & 1
\end{pmatrix}\right\} \\
\rer {a_3}_+ &=& \left\{\text{columns of }
\begin{pmatrix}
1 & 2 & 1 & 1 & 0 \\
0 & 1 & 1 & 2 & 1
\end{pmatrix}\right\} \\
\rer {a_4}_+ &=& \left\{\text{columns of }
\begin{pmatrix}
1 & 1 & 2 & 1 & 0 \\
0 & 1 & 3 & 2 & 1
\end{pmatrix}\right\} \\
\rer {a_5}_+ &=& \left\{\text{columns of }
\begin{pmatrix}
1 & 3 & 2 & 1 & 0 \\
0 & 1 & 1 & 1 & 1
\end{pmatrix}\right\}
\end{eqnarray*}
In this example, all real roots are positive or negative.
\end{Example}

\begin{Example} \label{ex:2002cont}
In the Cartan graph of Example 
\ref{ex:2002} we have $C^a=\begin{pmatrix}
    2& 0 \\ 0 & 2 \end{pmatrix}$ for
    all $a$. 
    According to Definition \ref{def:WG}
    this implies that 
    $\sigma_i^a(\alpha_j)=\alpha_j$
    for all $i\neq j$ (and as always
    $\sigma_i^a(\alpha_i)=-\alpha_i$). 
    Using Definition \ref{def:realroots}
    the real roots are for all $a$
    given by
    $$(R^{\mathrm{re}})^a = \left\{ \pm
    \begin{pmatrix} 1\\0 \end{pmatrix},
    \pm
    \begin{pmatrix} 0\\1 \end{pmatrix}
    \right\}.
    $$ 
\end{Example}

\begin{Example} \label{ex:root2cont}
In the Cartan graph of type $\Lambda_4$
of Example 
\ref{ex:root2} we have $C^a=\begin{pmatrix}
    2& -2\sqrt{2} \\ -\sqrt{2} & 2 \end{pmatrix}$ for
    all $a$. By Definition \ref{def:WG}
    we have
    $$\sigma_1^a(\begin{pmatrix} 0\\1
    \end{pmatrix} )= \begin{pmatrix}
        2\sqrt{2}\\1 \end{pmatrix}
        \mbox{~~~~and~~~~}
    \sigma_2^a(\begin{pmatrix} 1\\0
    \end{pmatrix} )= \begin{pmatrix}
        1 \\ \sqrt{2} \end{pmatrix}.   
    $$
    From this one can work out the set
    of real roots to be
    $$(R^{\mathrm{re}})^a = 
    \left\{ \pm \begin{pmatrix} 1\\0
    \end{pmatrix}, \pm \begin{pmatrix} 0\\1
    \end{pmatrix}, \pm \begin{pmatrix} 1\\\sqrt{2} \end{pmatrix}, \pm \begin{pmatrix} 2\sqrt{2}\\1
    \end{pmatrix}, \pm \begin{pmatrix} 3\\\sqrt{2} \end{pmatrix},
    \pm \begin{pmatrix} 2\sqrt{2}\\ 3 \end{pmatrix},\ldots
    \right\}.
    $$
    In particular, there are infinitely 
    many real roots for each $a$. 
\end{Example}

\medskip

\begin{Definition}[cf.\ {\cite[Def.\ 2.2]{p-CH09a}}] \label{def:rootsystem}
Let $\Cc =\Cc (I,A,(\rfl _i)_{i\in I},(\Cm ^a)_{a\in A})$ be a Cartan
graph of type $\Lambda_p$. For all $a\in A$ let $R^a\subseteq \Ri^I$, and define
$m_{i,j}^a= |R^a \cap (\mathbb{R}_{\ge 0} \alpha_i + \mathbb{R}_{\ge 0} \alpha_j)|$ for all $i,j\in
I$ and $a\in A$ ($m_{i,j}^a$ is possibly infinite).
We say that
\[ \rsC = \rsC (\Cc , (R^a)_{a\in A}) \]
is a \emph{root system of type} $\Cc$ if it satisfies the following
axioms.
\begin{enumerate}
\item[(R1)]
$R^a=R^a_+\cup - R^a_+$, where $R^a_+=R^a\cap \mathbb{R}_{\ge 0}^I$, for all
$a\in A$.
\item[(R2)]
$R^a\cap \mathbb{R}\alpha_i=\{\alpha_i,-\alpha_i\}$ for all $i\in I$, $a\in A$.
\item[(R3)]
$\s _i^a(R^a) = R^{\rfl _i(a)}$ for all $i\in I$, $a\in A$.
\item[(R4)]
If $i,j\in I$ and $a\in A$ such that $i\not=j$ and $m_{i,j}^a$ is
finite, then
$(\rfl _i\rfl _j)^{m_{i,j}^a}(a)=a$.
\end{enumerate}
\end{Definition}

The root system $\rsC $ is called \emph{finite} if for all $a\in A$ the set $R^a$ is finite.

\begin{Example}
    We consider the Cartan graph of Example 
    \ref{ex:2002}. The real roots at each vertex $a$ have been computed in 
    Example \ref{ex:2002cont} to be
    $$(R^{\mathrm{re}})^a = \left\{ \pm
    \begin{pmatrix} 1\\0 \end{pmatrix},
    \pm
    \begin{pmatrix} 0\\1 \end{pmatrix}
    \right\}.
    $$ 
    We check the axioms of a root system
    as in Definition \ref{def:rootsystem}. 
    (R1) and (R2) are clearly satisfied, and
    also (R3) because the $C^a$ are diagonal
    matrices. The numbers $m_{1,2}^a$ from
    Definition \ref{def:rootsystem} are in this example $m_{1,2}^a= 2$. Then
    $$(\rho_1\rho_2)^{m_{1,2}^0}(0)
    = (\rho_1\rho_2)^2(0) = (\rho_1\rho_2)(2)
    = 4 \neq 0.
    $$
    Thus, (R4) is not satisfied and the real
    roots do not form a root system in the
    sense of Definition \ref{def:rootsystem}. 
\end{Example}

\begin{Example}
    We consider the Cartan graph of Example 
    \ref{ex:root2}. The real roots at each vertex $a$ have been computed in 
    Example \ref{ex:root2cont} to be
    $$(R^{\mathrm{re}})^a =
    \left\{ \pm \begin{pmatrix} 1\\0
    \end{pmatrix}, \pm \begin{pmatrix} 1\\\sqrt{2} \end{pmatrix}, \pm \begin{pmatrix} 2\sqrt{2}\\1
    \end{pmatrix}, \pm \begin{pmatrix} 3\\\sqrt{2} \end{pmatrix},
    \pm \begin{pmatrix} 2\sqrt{2}\\ 3 \end{pmatrix},\ldots
    \right\}.
    $$ 
    We check the axioms of a root system
    as in Definition \ref{def:rootsystem}. 
    (R1) and (R2) are clearly satisfied, and
    also (R3) because the sets of real roots
    are the same for each vertex. The numbers $m_{1,2}^a$ from
    Definition \ref{def:rootsystem} are in this example $m_{1,2}^a= \infty$. 
    Thus, (R4) is trivially satisfied. Therefore,
    the real
    roots form a root system in the
    sense of Definition \ref{def:rootsystem}. 
\end{Example}

\subsection{Propagation on infinite frieze patterns of type $\Lambda_p$}

The following matrices play a crucial role for 
determining entries in infinite frieze patterns.

\begin{Definition}
For any number $c$ we set
$\mu(c) = \begin{pmatrix} 0 & -1 \\ 1 & c \end{pmatrix}.
$
\end{Definition}

Then we have the following special case of 
the Ptolemy relations in Proposition
\ref{prop:Ptolemy}. 

\begin{Proposition}
Let $\mathcal{F}=(f_{i,j})$ be an infinite frieze pattern. Then for all $i\in \mathbb{Z}$ and all $j\ge i+2$ we have
\begin{equation} \label{eq:propagation}
(f_{i,j-1},f_{i,j}) \mu(f_{j-1,j+1}) =
(f_{i,j-1},f_{i,j}) \begin{pmatrix} 0 & -1 \\ 1 & f_{j-1,j+1} \end{pmatrix}
= (f_{i,j},f_{i,j+1}).
\end{equation}
\end{Proposition}

\begin{proof} 
This is Proposition \ref{prop:Ptolemy} 
for the case $i<j-1<j<j+1$ (use that 
$f_{j-1,j}=1=f_{j,j+1}$ by Definition 
\ref{def:infinitefrieze}).
\end{proof}

This propagation formula for (infinite) frieze 
patterns motivates the following 
more general construction of patterns of numbers.
Namely, we start with an arbitrary sequence
and produce numbers inductively by the formula
(\ref{eq:propagation}). 

\begin{Definition} \label{def:halfpattern}
Let 
$(f_{i,i+2})_{i\in \mathbb{Z}}$ be a sequence
of real numbers. Setting $f_{i,i}=0$ and 
$f_{i,i+1}=1$ for all $i\in \mathbb{Z}$, 
numbers $f_{i,j+1}$ for $i\in \mathbb{Z}$ and
$j\ge i+2$ are defined inductively by the equations
(\ref{eq:propagation}).
\end{Definition}

The numbers resulting from the construction 
in Definition \ref{def:halfpattern} can be 
arranged as usual:
$$\begin{array}{ccccccccccccc}
& & \iddots & & \iddots & & \iddots & & \iddots & & \iddots & & \\
\cdots &  & f_{-2,2} & & f_{-1,3} & & f_{0,4} & & f_{1,5} & & f_{2,6} & & \cdots \\
& f_{-2,1}& & f_{-1,2} & & f_{0,3} & & f_{1,4} & & f_{2,5} & &f_{3,6} &  \\
\cdots & & f_{-1,1} & & f_{0,2} & & f_{1,3} & & f_{2,4} & & f_{3,5} &  & \cdots \\
  & 1 & & 1 & & 1 & & 1 & & 1 & & 1 & \\
  \cdots  &  & 0 &  & 0 &  & 0 &  & 0 &  & 0 & & \cdots \\
 \end{array}
$$

\begin{Proposition} \label{prop:unimodular}
Let $(f_{i,i+2})_{i\in \mathbb{Z}}$ be a sequence
of real numbers and let 
$(f_{i,j})_{i\in \mathbb{Z},j\ge i}$ be the 
numbers defined inductively as in Definition 
\ref{def:halfpattern}. Then for all 
$i\in \mathbb{Z}$
and $j\ge i+2$ we have
$f_{i,j}f_{i+1,j+1} - f_{i,j+1}f_{i+1,j}=1
$
that is, each diamond 
$$\begin{array}{ccc}
& f_{i,j+1} & \\ f_{i,j} & & f_{i+1,j+1} \\  
& f_{i+1,j} & 
\end{array}
$$
in the above array of numbers 
satisfies the unimodular rule. 
\end{Proposition}

\begin{proof}
We fix $i\in \mathbb{Z}$ and 
proceed by induction on $j$. For $j=i+2$, we get
from equation (\ref{eq:propagation}) that  
$$f_{i,i+3} = -1 + f_{i,i+2}f_{i+1,i+3}
$$
which is equivalent to the diamond
$$\begin{array}{ccc}
& f_{i,i+3} & \\ f_{i,i+2} & & f_{i+1,i+3} \\  
& 1 & 
\end{array}
$$
satisfying the unimodular rule. Inductively, we can 
assume that the diamond 
$$\begin{array}{ccc}
& f_{i,j} & \\ f_{i,j-1} & & f_{i+1,j} \\  
& f_{i+1,j-1} & 
\end{array}
$$
satisfies the unimodular rule, that is, 
\begin{equation}
 f_{i,j-1}f_{i+1,j}-f_{i,j}f_{i+1,j-1} = 1.
\end{equation}
We have to consider the next diamond to the top right,
$$\begin{array}{ccc}
& f_{i,j+1} & \\ f_{i,j} & & f_{i+1,j+1} \\  
& f_{i+1,j} & 
\end{array}
$$
By Definition \ref{def:halfpattern} and 
(\ref{eq:propagation}) we have that 
$$f_{i,j+1} = -f_{i,j-1}+f_{i,j}f_{j-1,j+1}
\mbox{\,~~~and\,~~~}
f_{i+1,j+1} = -f_{i+1,j-1}+f_{i+1,j}f_{j-1,j+1}.
$$
From this we obtain
\begin{eqnarray*}
f_{i,j}f_{i+1,j+1} - f_{i,j+1}f_{i+1,j}
& = & f_{i,j}(-f_{i+1,j-1}+f_{i+1,j}f_{j-1,j+1})
- (-f_{i,j-1}+f_{i,j}f_{j-1,j+1})f_{i+1,j} \\
& = & -f_{i,j}f_{i+1,j-1} + f_{i,j-1}f_{i+1,j} 
\,=\, 1 
\end{eqnarray*}
where the last equality holds by induction 
hypothesis. This means that the diamond under
consideration satisfies the unimodular rule.
\end{proof}

The entries in the above array of numbers satisfy
numerous further relations, called {\em Ptolemy
relations}. Note that the following result is known
to hold for frieze patterns (see
Proposition \ref{prop:Ptolemy}), 
but we provide a 
separate proof as for our arrays of numbers we a 
priori only assume that the entries are given by 
the propagation formula, see Definition 
\ref{def:halfpattern}. However, the following proof 
is very similar to the proof of 
\cite[Proposition 3.3]{CHJ20}.

\begin{Proposition} \label{prop:infinitePtolemy}
Let $(f_{i,i+2})_{i\in \mathbb{Z}}$ be a sequence
of real numbers and let 
$(f_{i,j})_{i\in \mathbb{Z},j\ge i}$ be the 
numbers defined inductively as in Definition 
\ref{def:halfpattern}. Then for all 
$i\in \mathbb{Z}$ and 
$i\le j\le k\le\ell$
we have 
$$f_{i,k}f_{j,\ell} = f_{i,\ell}f_{j,k} + 
f_{i,j}f_{k,\ell}.
$$
\end{Proposition}

\begin{proof}
The equation clearly holds if one of the inequalities 
in $i\le j\le k\le \ell$ is an equality (as $f_{r,r}=0$ for all $r\in \mathbb{Z}$). So we can assume that
$i<j<k<\ell$. For any $r\le s$ we set
$$M_{r+1,s}:= \prod_{t=r+1}^s \mu(f_{t-1,t+1}).
$$
From formula (\ref{eq:propagation}) defining the 
entries of our array of numbers, we can deduce that
\begin{equation} \label{eq:Mformula}
M_{r+1,s} = \begin{pmatrix} 
-f_{r+1,s} & & - f_{r+1,s+1} \\
f_{r,s} & & f_{r,s+1} 
\end{pmatrix}.
\end{equation}
By our assumption $i<j<k<\ell$ we have the following 
matrices and
factorisations: 
$M_{i+1,k} = M_{i+1,j}M_{j+1,k}$, $M_{j+1,\ell} = M_{j+1,k}M_{k+1,\ell}$, 
and $M_{i+1,\ell} = M_{i+1,j} M_{j+1,k}M_{k+1,\ell}$.
Using (\ref{eq:Mformula}) and  
considering the $(2,1)$-entries of these 
matrices we get that 
\begin{eqnarray}
\label{me1} f_{i,k} &=& f_{i,j+1} f_{j,k} - f_{j+1,k} f_{i,j} \\
\label{me2} f_{j,\ell} &=& f_{j,k+1} f_{k,l} - f_{k+1,\ell} f_{j,k} \\
\label{me3} f_{i,\ell} &=& f_{i,j+1} f_{j,k+1} f_{k,\ell} - f_{i,j+1} f_{k+1,\ell} f_{j,k} + f_{j+1,k} f_{k+1,\ell} f_{i,j} - f_{j+1,k+1} f_{i,j} f_{k,\ell}.
\end{eqnarray}
From these equations we can conclude
(use (\ref{me1}) and (\ref{me2}) for the first
equality and (\ref{me3}) for the second equality):
\begin{eqnarray*}
f_{i,j} f_{k,\ell} + f_{j,k} f_{i,\ell} - f_{i,k} f_{j,\ell}
&=& (f_{i,j} f_{k,\ell} + f_{j,k} f_{i,\ell}) - (f_{i,j+1} f_{j,k} - f_{j+1,k} f_{i,j})( f_{j,k+1} f_{k,l} - f_{k+1,\ell} f_{j,k})\\
&=& (f_{i,j} f_{k,\ell} + f_{j,k} f_{i,\ell}) - (f_{i,j+1} f_{j,k} - f_{j+1,k} f_{i,j})( f_{j,k+1} f_{k,l} - f_{k+1,\ell} f_{j,k})  \\
&& +f_{j,k} ( f_{i,j+1} f_{j,k+1} f_{k,\ell} - f_{i,j+1} f_{k+1,\ell} f_{j,k} + f_{j+1,k} f_{k+1,\ell} f_{i,j} \\ 
&& - f_{j+1,k+1} f_{i,j} f_{k,\ell}-  f_{i,\ell}) \\
&=& f_{i,j} f_{k,\ell} \left(1- \left(f_{j,k}f_{j+1,k+1} - f_{j,k+1} f_{j+1,k}\right)\right) = 0
\end{eqnarray*}
where the last equality holds by Proposition 
\ref{prop:unimodular}. 
\end{proof}

Note that the patterns of numbers defined
by Definition \ref{def:halfpattern} are in
general not infinite frieze patterns, as 
the latter require all entries to be 
positive real numbers, see Definition 
\ref{def:infinitefrieze}.

\subsection{Reflection groupoids of rank two and frieze patterns}

The aim of this section is to explain the 
connection between infinite frieze patterns 
of type $\Lambda_p$ and certain Cartan graphs 
of type $\Lambda_p$. 

\begin{Definition}
If $\Cc $ is a Cartan graph and there exists a root system of type $\Cc $, then we say that $\Cc$ \emph{permits} a root system.
\end{Definition}

\begin{Lemma}\label{lemC3}
If $\Cc =\Cc (I,A,(\rfl _i)_{i\in I},(\Cm ^a)_{a\in A})$ is a Cartan graph of type $\Lambda_p$ of rank $2$ permitting a root system, then
\begin{itemize}
  \item [(C3)] if $a,b\in A$ and the identity
  as linear map is in $\Hom (a,b)$, then $a=b$.
\end{itemize}
\end{Lemma}
\begin{proof}
Assume $\id\in\Hom(a,b)$. We show that $a=b$ or $|R^a|$ is finite (as we will see, this then also implies $a=b$).
Without loss of generality, let $\id = \id^b \s_i \cdots \s_1 \s_2 \s^a_1$ be a product of $\ell$ reflections starting with a reflection labeled $1$ and ending with label $i=1$ or $i=2$ depending on whether $\ell$ is odd or even (in fact it will turn out that $\ell$ is always even).
We assume $\ell>0$ because otherwise $a=b$ anyway.

Denote $a_0:=a,a_1,\ldots,a_\ell:=b$ the objects on the way from $a$ to $b$,
and let $\varphi_k:=\id^b \s_i \cdots \s_j^{a_k}$ be the product from $a_k$ to $b$ (for the correct $j\in\{1,2\}$). In particular,
$\varphi_0$ and 
$\varphi_{\ell}$ are identity maps.
We claim that
\begin{equation} \label{eq:Rb}
R^b = \{\pm\varphi_k(\alpha_\nu) \mid \nu=1,2,\:\: k=0,\ldots,\ell\}. 
\end{equation}
The inclusion $\supseteq$ follows from (R2) and (R3). 
Let $\beta\in R^b$.
For each $k\in \{0,\ldots,\ell\}$ we have two simplicial cones
$$
C^k_1:=\langle\varphi_k(-\alpha_1),\varphi_k(\alpha_2)\rangle_{\ge 0},
\quad
C^k_2:=\langle\varphi_k(\alpha_1),\varphi_k(-\alpha_2)\rangle_{\ge 0}
$$
(the notation $\langle .,.\rangle_{\ge 0}$
denotes linear combinations with non-negative
coefficients). 

By (R2) and (R3), the boundaries of the
cones $C_1^k$ and $C_2^k$ are roots. 
Moreover, the interior of such a cone cannot 
contain any roots. (In fact, applying the 
sequence $\varphi_k^{-1}$ of reflections 
to a linear combination
in the interior of the cone yields by (R3)
a root having coordinates with different
signs, contradicting (R1).) As a consequence, 
for the roots in $R^b$ we have 
$|C^k_1\cap R^b|=|C^k_2\cap R^b|=2$. 
As another consequence, we also have 
that
$$ (C^k_1\cup C^k_2) \cap (C^{k+1}_1\cup C^{k+1}_2) = \{ \pm \varphi_k(\alpha_\nu)\}$$
for the correct choice of $\nu\in \{1,2\}$.

Note that by definition, $C_1^k=-C_2^k$, i.e.
the cones 
$C_1^k$ and $C_2^k$ are opposite to each other. 
Since $\varphi_k \sigma_{\nu}^{a_k}=
\varphi_{k+1}$ it also follows from the 
definition of the cones that 
one of $\pm \varphi_k(\alpha_{\nu})$
is a common boundary of $C_1^k$ and $C_2^{k+1}$
(and the other is a common boundary of
$C_2^k$ and $C_1^{k+1}$), that is, 
the cones $C_1^k$ and $C_2^{k+1}$ are
adjacent (and also $C_2^k$ and $C_1^{k+1}$
are adjacent).
Thus each object on the way between $a$ and $b$ yields a pair of cones which contain only simple roots; for two adjacent objects, these cones are adjacent.
But $\varphi_0=\varphi_\ell=\id$, thus we obtain (note that $\ell>0$ and that $C_1^0$ 
and $C_2^0$ cover the second and fourth 
quadrant)
\begin{equation}\label{eq:CR2}
\bigcup_{k=0}^\ell C_1^k\cup C_2^k = \mathbb{R}^2.
\end{equation}
As a consequence, $\beta$ is a multiple of $\varphi_k(\alpha_\nu)$ for some $\nu,k$; by (R2) 
and (R3) we get $\beta=\pm\varphi_k(\alpha_\nu)$, thus proving
the remaining inclusion in the above claim
(\ref{eq:Rb}).

Equation (\ref{eq:CR2}) shows that the whole
space $\mathbb{R}^2$ is a union of finitely
many cones and the roots appear on the 
walls of these cones. Counting the number 
of cones we get that 
$|R^b|=|R^a|$ is finite and 
$$\ell=|R^a|=2|R^a \cap (\mathbb{R}_{\ge 0} \alpha_i + \mathbb{R}_{\ge 0} \alpha_j)|$$ 
so by (R4), $a=b$.
\end{proof}

\begin{Remark}
For $p=3$, Lemma \ref{lemC3} also follows from \cite[Theorem 1]{p-HY08}.
The statement in Lemma \ref{lemC3} is also true for arbitrary Cartan graphs of arbitrary rank, but rank $2$ is sufficient for our purposes.
\end{Remark}

As a graph, a connected Cartan graph of rank two is a connected 2-regular graph.
Thus it
can only be one of the following:
an infinite string, an infinite string with one end (with a loop attached), 
a cycle, or a chain (a finite string with two ends with loops attached, as in 
Example \ref{ex:13122}). In each case, we obtain a sequence of Cartan entries labeled by $\mathbb{Z}$:

\begin{Definition} \label{def:quiddity}
Let $\Cc=\Cc (I=\{1,2\},A,(\rfl _i)_{i\in I},(C^a)_{a\in A})$ be a connected Cartan graph of type $\Lambda_p$ of rank two and $a_0\in A$ an object.
For $i>0$ let
$$ a_i := \begin{cases}
\rfl_1 (a_{i-1}) & i \text{ odd},\\
\rfl_2 (a_{i-1}) & i \text{ even},
\end{cases}
\quad
a_{-i} := \begin{cases}
\rfl_2 (a_{1-i}) & i \text{ odd},\\
\rfl_1 (a_{1-i}) & i \text{ even}.
\end{cases}
$$
Note that the $a_i$ are not necessarily distinct.
Define
$$
\qs_i :=
\begin{cases}
-\cm^{a_i}_{12} & i \text{ odd}, \\
-\cm^{a_i}_{21} & i \text{ even}.
\end{cases}
$$
We call $(q_i)_{i\in\mathbb{Z}}$ the \emph{quiddity sequence} of $\Cc$ at $a_0$.
\end{Definition}

\begin{Example}
For the Cartan graph from Example \ref{ex:13122}, we obtain the quiddity sequence
$$
\ldots,1,3,1,2,2,1,3,1,2,2,1,3,1,2,2,\ldots
$$
The Cartan graph of type $\Lambda_p$
from Example \ref{ex:2002} yields the 
quiddity sequence
$$\ldots,0,0,0,0,0,0,0,\ldots.
$$
The Cartan graph of type $\Lambda_4$
from Example \ref{ex:root2} yields the 
quiddity sequence
$$\ldots,\sqrt{2},2\sqrt{2},\sqrt{2},
2\sqrt{2},\sqrt{2},2\sqrt{2},\ldots.
$$
The Cartan graph from Example \ref{affine41} gives
$$
\ldots,1,4,1,4,1,4,\ldots
$$
The Cartan graph of type $\Lambda_p$ 
from Example \ref{dihedralgroup} yields 
the quiddity sequence
$$\ldots, \lambda_p,\lambda_p,
\lambda_p,\lambda_p,\lambda_p,\lambda_p,
\ldots
$$
\end{Example}

\begin{Remark} \label{rem:a0a1}
Starting with an object $a_0\in A$ (in blue below), the sequence 
$(a_i)_{i\in \mathbb{Z}}$ is by Definition
\ref{def:quiddity}
given by
\begin{equation} \label{eq:a0}
(\ldots, \rho_2\rho_1\rho_2(a_0),\rho_1\rho_2(a_0),\rho_2(a_0),{\color{blue} a_0},\rho_1(a_0), \rho_2\rho_1(a_0), 
\rho_1\rho_2\rho_1(a_0),\ldots).
\end{equation}
The corresponding quiddity sequence 
$(q_i)_{i\in \mathbb{Z}}$ at $a_0$ is then 
given by 
\begin{equation} \label{eq:quidditya0}
    (\ldots, -c_{12}^{\rho_2\rho_1\rho_2(a_0)},-c_{21}^{\rho_1\rho_2(a_0)},-c_{12}^{\rho_2(a_0)},
    {\color{blue} -c_{21}^{a_0}}, -c_{12}^{\rho_1(a_0)}, -c_{21}^{\rho_2\rho_1(a_0)}, 
-c_{12}^{\rho_1\rho_2\rho_1(a_0)},\ldots).
\end{equation}
If we instead start with the object $\rho_1(a_0)$
(again in blue), then we get the sequence
$$(\ldots,\rho_1(a_2)=a_3,\rho_2\rho_1(a_0)=a_2,{\color{blue} \rho_1(a_0)=a_1},\rho_1^2(a_0)=a_0, \rho_2(a_0)=a_{-1},\ldots)
$$
which is obtained from the sequence (\ref{eq:a0})
by reversing and shifting.
The corresponding quiddity sequence has the form
$$(\ldots,-c_{21}^{a_3},-c_{12}^{a_2},
{\color{blue} -c_{21}^{a_1}},-c_{12}^{a_0}, 
-c_{21}^{a_{-1}},-c_{12}^{a_{-2}}\ldots). 
$$
If we compare this to the quiddity sequence
(\ref{eq:quidditya0}), this looks completely 
different (for instance, here $-c_{12}^{a_0}$
appears instead of $-c_{21}^{a_0}$ above). 
However, by the definition of a Cartan graph (see
Definition \ref{def:cartangraph}\,(C2)) this sequence 
is equal to
$$(\ldots,-c_{21}^{a_4},-c_{12}^{a_3},
{\color{blue} -c_{21}^{a_2}},-c_{12}^{a_1}, 
-c_{21}^{a_{0}},-c_{12}^{a_{-1}}\ldots). 
$$
And now we see that this sequence can be obtained
from (\ref{eq:quidditya0}) by reversing and shifting.
\end{Remark}

\begin{Definition} \label{def:infinitedihedral}
We write $\Dih_\infty$ 
for the infinite dihedral group acting on $\mathbb{Z}$ with generators $a\mapsto -a$, $a \mapsto a+1$. We say that two sequences
$(q_i)_{i\in\mathbb{Z}}$, $(q'_i)_{i\in\mathbb{Z}}$ are \emph{equivalent} if there exists
$\sigma\in\Dih_\infty$ such that $q_{\sigma(i)} = q'_{i}$ for all $i\in\mathbb{Z}$.
\end{Definition}

\begin{Proposition}\label{wg:quideq}
There is a 1-1 correspondence between connected simply connected Cartan graphs of type $\Lambda_p$ of rank two with infinitely many 
vertices up to equivalence and sequences $(q_i)_{i\in\mathbb{Z}}$ with
$q_i\in \mathbb{Z}_{\ge 0} \lambda_p$
up to equivalence.
\end{Proposition}

\begin{proof}
Let $\Cc$ be a connected simply connected Cartan graph of type $\Lambda_p$ 
of rank two with infinitely many objects. Then its quiddity sequence $(q_i)_{i\in\mathbb{Z}}$ at any object $a_0$ satisfies $q_i\in 
\mathbb{Z}_{\ge 0}\lambda_p$ for all $i\in\mathbb{Z}$ by Definitions \ref{def:quiddity}
and \ref{def:genCartan}.
As a graph, $\Cc$ is an infinite string because it is simply connected.
Thus since $\Cc$ is connected, choosing another object than $a_0$ will give an equivalent sequence
(see Remark \ref{rem:a0a1}). 
Moreover, an equivalent Cartan graph will also produce an equivalent sequence: bijections $\varphi_0,\varphi_1$ as in the definition of equivalence of Cartan graphs (cf. Definition \ref{def:Cartanequiv})
would only shift or reverse the quiddity sequence. 
(The details are similar to Remark \ref{rem:a0a1}
and are left to the reader.) Thus, we have a 
well-defined map from equivalence classes of
connected simply connected Cartan graphs of
type $\Lambda_p$ of rank two to equivalence 
classes of sequences with entries in $Z_p$. 

For the converse, given a sequence $(q_i)_{i\in\mathbb{Z}}$ with 
$q_i\in \mathbb{Z}_{\ge 0}\lambda_p$, we construct a Cartan graph of rank two
$\Cc=\Cc (I=\{1,2\},A=\mathbb{Z},(\rfl _i)_{i\in I},(C^a)_{a\in A})$ where we set
$$
\rho_1(i) := \begin{cases}
i+1 & i \text{ even}, \\
i-1 & i \text{ odd},
\end{cases}
\quad
\rho_2(i) := \begin{cases}
i-1 & i \text{ even}, \\
i+1 & i \text{ odd}
\end{cases}
$$
and define the generalized Cartan matrices as
$$
C^a = \left\{ \begin{array}{ll}
\begin{pmatrix} 2 & -q_a \\ -q_{a+1} & 2 
\end{pmatrix} & \textrm{if $a$ odd} \\
\begin{pmatrix} 2 & -q_{a+1} \\ -q_{a} & 2 
\end{pmatrix} & \textrm{if $a$ even} 
\end{array}
\right.
$$
It is straightforward to check that conditions
(C1) and (C2) of Definition \ref{def:cartangraph}
are satisfied.
This Cartan graph is clearly connected, and has $(q_i)_{i\in\mathbb{Z}}$ as quiddity sequence
(cf. Definition \ref{def:quiddity}).
From the definition of the involutions $\rho_i$
it follows that $\Cc$ is simply connected
(cf. Definition \ref{def:simplyconnected}).
We also have to show that if we start with 
equivalent sequences $(q_i)_{i\in \mathbb{Z}}$
and $(q'_i)_{i\in \mathbb{Z}}$, that the 
corresponding Cartan graphs are equivalent (in the
sense of Definition \ref{def:Cartanequiv}).
For this it suffices to consider equivalent 
sequences resulting from the generators 
of the infinite dihedral group as in Definition 
\ref{def:infinitedihedral}. For the shift 
$q_i'=q_{i+1}$, 
an equivalence of the corresponding
Cartan graphs is given by the maps $\varphi_0$
interchanging 1 and 2, and $\varphi_1:a\mapsto 
a-1$. For the reversion $q_i'=q_{-i}$,
an equivalence of the corresponding
Cartan graphs is given by the maps $\varphi_0=
\mathrm{id}$, and $\varphi_1:a\mapsto -(a+1)$.
We leave the details of the verification to
the reader. Thus, we now have a well-defined
map from equivalence classes of sequences 
over $Z_p$ to equivalence classes of 
connected simply connected Cartan graphs of
type $\Lambda_p$ of rank two. 
\smallskip

Finally, from the above constructions it is 
straightforward to check that the two maps 
are inverse to each other (actually, they are
just inverses, not only up to equivalence). 
This completes the proof of the bijection
stated in the proposition. 
\end{proof}
\medskip

\begin{Remark} \label{rem:rootsinfrieze}
Let $\mathcal{C}$ be a connected 
Cartan graph of type
$\Lambda_p$ of rank 2 and let $a_0\in A$. 
According to Definition \ref{def:quiddity}
the Cartan graph around $a_0$ can be visualized
as 
\begin{equation} \label{fig:Cartan}
\ldots C^{a_{-2}} \stackrel{\rho_1}{\longleftrightarrow}
C^{a_{-1}} \stackrel{\rho_2}{\longleftrightarrow}
C^{a_0} \stackrel{\rho_1}{\longleftrightarrow}
C^{a_1} 
\stackrel{\rho_2}{\longleftrightarrow}
C^{a_2} \ldots
\end{equation}
(where the $a_i$ are not necessarily distinct). 
Moreover, also by Definition \ref{def:quiddity}
we have the quiddity cycle 
$(q_i)_{i\in \mathbb{Z}}$ with 
$$q_i= \left\{ \begin{array}{ll}
-c_{12}^{a_i} & i~\mathrm{odd}, \\
-c_{21}^{a_i} & i~\mathrm{even}.
\end{array} \right.
$$
Up to translation, 
we can set $q_i=f_{i,i+2}$ for $i\in \mathbb{Z}$. 
By Definition \ref{def:halfpattern}
and Proposition \ref{prop:unimodular}
this quiddity cycle defines an array of numbers
$(f_{i,j})$ such that each diamond satisfies 
the unimodular rule. 
On the other hand, we
have by Definition \ref{def:realroots}
the set of real roots at the object $a_0$ given 
by 
\[ \rer {a_0}=\{ \id ^{a_0} \s _{i_1}\cdots \s_{i_k}(\alpha_j) \mid
k\in \mathbb{N} _0,\:\: i_1,\dots,i_k,j\in I\}\subseteq \Ri^I.\]
We want to explain that the (positive) 
real roots appear in the array of numbers
and where precisely they are located.
Note that the real roots at $a_0$ 
are obtained by 
applying to the unit vectors 
all products of reflections 
corresponding to paths in (\ref{fig:Cartan})
ending at $a_0$. 

$$\begin{array}{ccccccccccccc}
&& & & \iddots & & \iddots & & \iddots & & \iddots & & \\
&& & {\color{red} f_{-3,2}}& & f_{-2,3} & & f_{-1,4} & & f_{0,5} & & {\color{green} f_{1,6}} &  \\
 \cdots & & {\color{red}f_{-3,1}} &  & 
 {\color{red} f_{-2,2}} & & f_{-1,3} & & f_{0,4} & & {\color{green} f_{1,5}} & &  \cdots \\
&& & {\color{red} f_{-2,1}} & & {\color{red} f_{-1,2}} & & f_{0,3} & & 
{\color{green} f_{1,4}} & & {\color{green} f_{2,5}} &  \\
 \cdots & & q_{-2} & & {\color{red} q_{-1}} & & {\color{red} q_0} & & 
 {\color{green} q_1} & & {\color{green} q_2} & & \cdots \\
&&   & 1 & & {\color{red} 1} & & {\color{olive} 1} & & 
{\color{green} 1} & & 1 & \\
\cdots && 0 &  & 0 &  & {\color{red} 0} &  & {\color{green} 0} &  & 0 &  & 0  \\
 \end{array}
$$
The vectors $\begin{pmatrix} 1\\0 \end{pmatrix}$,
$\begin{pmatrix} q_1\\1 \end{pmatrix}$,
$\begin{pmatrix} f_{1,4}\\q_2 \end{pmatrix}$,
$\begin{pmatrix} f_{1,5}\\f_{2,5}\end{pmatrix}$,
$\ldots$ highlighted in green are the
real roots of the 
form $\mathrm{id}\begin{pmatrix} 1\\0
\end{pmatrix}$, $\sigma_1^{a_1}\begin{pmatrix} 0\\1
\end{pmatrix}$, $\sigma_1^{a_1}\sigma_2^{a_2} \begin{pmatrix} 1\\0\end{pmatrix}$,
$\sigma_1^{a_1}\sigma_2^{a_2}\sigma_1^{a_3} \begin{pmatrix} 0\\1\end{pmatrix}$, $\ldots$
(corresponding to paths in \ref{fig:Cartan}
coming from the right). 

The vectors $\begin{pmatrix} 0\\1 \end{pmatrix}$,
$\begin{pmatrix} 1\\q_0 \end{pmatrix}$,
$\begin{pmatrix} q_{-1}\\f_{-1,2} \end{pmatrix}$,
$\begin{pmatrix} f_{-2,1}\\f_{-2,2}\end{pmatrix}$,
$\ldots$ highlighted in red are the real roots of the 
form $\mathrm{id}\begin{pmatrix} 0\\1
\end{pmatrix}$, $\sigma_2^{a_{-1}}\begin{pmatrix} 1\\0
\end{pmatrix}$, $\sigma_2^{a_{-1}}\sigma_1^{a_{-2}} \begin{pmatrix} 0\\1\end{pmatrix}$,
$\sigma_2^{a_{-1}}\sigma_1^{a_{-2}}\sigma_2^{a_{-3}} \begin{pmatrix} 1\\0\end{pmatrix}$, $\ldots$
(corresponding to paths in \ref{fig:Cartan}
coming from the left). 

However, these are not all real roots at
the object $a_0$, the remaining ones are 
the negative vectors of the vectors above. 
All real roots can conveniently be 
visualized if one extends the above array
of numbers with a copy with minus signs 
by flipping the array at the line with zeros: 
$$\begin{array}{cccccccccccc}
& & & \iddots & & \iddots & & \iddots & & \iddots & & \\
& & {\color{red} f_{-3,2}}& & f_{-2,3} & & f_{-1,4} & & f_{0,5} & & {\color{green} f_{1,6}} &  \\
 \cdots  & {\color{red}f_{-3,1}} &  & 
 {\color{red} f_{-2,2}} & & f_{-1,3} & & f_{0,4} & & {\color{green} f_{1,5}} & &  \cdots \\
& & {\color{red} f_{-2,1}} & & {\color{red} f_{-1,2}} & & f_{0,3} & & 
{\color{green} f_{1,4}} & & {\color{green} f_{2,5}} &  \\
 \cdots  & q_{-2} & & {\color{red} q_{-1}} & & {\color{red} q_0} & & 
 {\color{green} q_1} & & {\color{green} q_2} & & \cdots \\
&   & 1 & & {\color{red} 1} & & {\color{olive} 1} & & 
{\color{green} 1} & & 1 & \\
\cdots & 0 &  & 0 &  & {\color{red} 0} &  & {\color{green} 0} &  & 0 &  & \cdots  \\
&  & -1 & & {\color{red} -1} & & {\color{olive} -1} & & 
{\color{green} -1} & & -1 & \\
 \cdots  & -q_{-2} & & {\color{red} -q_{-1}} & & {\color{red} -q_0} & & 
 {\color{green} -q_1} & & {\color{green} -q_2} & & \cdots \\
& & {\color{red} -f_{-2,1}} & & 
{\color{red} -f_{-1,2}} & & -f_{0,3} & & 
{\color{green} -f_{1,4}} & & 
{\color{green} -f_{2,5}} &  \\
 \cdots  & {\color{red} -f_{-3,1}} &  & 
 {\color{red} -f_{-2,2}} & & -f_{-1,3} & & -f_{0,4} & & {\color{green} -f_{1,5}} & &  \cdots \\
& & {\color{red} -f_{-3,2}}& & -f_{-2,3} & & 
-f_{-1,4} & & -f_{0,5} & & 
{\color{green} -f_{1,6}} &  \\
& \iddots & & \iddots & & \iddots & & \iddots & & & & \\
 \end{array}
$$
Then all real roots at the object $a_0$ 
are given by the vectors in the coloured
area (with entries ordered suitably, as
described above). 

Note that up to signs, all real roots
occur as columns in the infinite matrix
given by the two coloured diagonals 
from bottom left to top right. 
\end{Remark}

Note that for an arbitrary reflection groupoid, the real roots are not necessarily a root system.
We now clarify how this relates to infinite frieze patterns.

\begin{Theorem} \label{thm:cartanfriezep}
There is a 1-1 correspondence between connected simply connected Cartan graphs of type $\Lambda_p$ of rank two with infinitely many 
vertices permitting a root system
up to equivalence and (tame) infinite frieze patterns of type $\Lambda_p$ up to equivalence.
\end{Theorem}

\begin{proof}
Assume that we have an infinite frieze pattern 
$\mathcal{F}$ of type $\Lambda_p$.
The quiddity sequence of $\mathcal{F}$
corresponds by 
Proposition \ref{wg:quideq}
to a connected simply connected Cartan graph $\Cc$.
By Remark \ref{rem:rootsinfrieze}, the real 
roots at an object are given by the columns of 
the infinite matrices given by two 
consecutive diagonals of $\mathcal{F}$
(up to signs, to be precise
the two diagonals give half
of the real roots). We have to show that 
the real roots form a root system, i.e. 
that they satisfy the axioms (R1)-(R4)
of Definition \ref{def:rootsystem}.
\smallskip

All entries in the infinite frieze pattern
$\mathcal{F}$ are positive (by 
Definition \ref{def:infinitefrieze}), so (R1) is satisfied (cf. Remark {\ref{rem:rootsinfrieze}).
Moreover, we have $0$'s only on the bottom line
of $\mathcal{F}$, thus also (R2) is satisfied.

From the description of real roots in
Remark \ref{rem:rootsinfrieze} it is 
apparent that the reflections $\sigma_i$ 
map real roots to real roots, so that
axiom (R3) is satisfied. 

It remains to consider (R4). We consider
some $m_{1,2}^a$, that is, the number of 
real roots at the object $a$ with 
non-negative coordinates. 
By Remark \ref{rem:rootsinfrieze}, 
these real roots appear as columns in 
infinite matrices given by two consecutive 
diagonals. The 
determinant of the matrix consisting of 
any of these two columns is an entry in the
infinite frieze pattern $\mathcal{F}$
(this is true by Proposition 
\ref{prop:infinitePtolemy}, it is a Ptolemy relation
where two of the indices are consecutive),
hence is non-zero. In particular, 
the different columns of the infinite 
matrices must be different. Therefore,
$m_{1,2}^a=\infty$ for
all $a$ and (R4) holds 
trivially. 
\smallskip

For the converse, let $\Cc$ be a Cartan graph with all the assumptions given in the theorem.
We consider its quiddity sequence 
as in Definition \ref{def:quiddity}. 
The quiddity sequence defines an infinite 
array $\mathcal{F}$ of numbers by Definition 
\ref{def:halfpattern}. For this to become
an infinite frieze pattern, all entries
must be positive. 
Assume that $\mathcal{F}_{i,j}=0$ for some $i<j$. Since adjacent $2\times 2$-determinants are equal to $1$, the $2\times 2$-submatrix with upward diagonals labeled $i,i+1$ and 
downward diagonals labeled $j,j+1$ 
is of the form
$$
\begin{matrix}
0 & \pm 1 \\
\mp 1 & x
\end{matrix}
$$
for some $x$. The second row and the second column of this matrix are roots of the root system of $\Cc$ (at certain objects). By (R1) this implies $x=0$.
But then propagation along rows starting at $(i,i)$ gives a morphism between two different objects which is $\id$ as a linear map. 
With (C3) (Lemma \ref{lemC3}) this is impossible since $\Cc$ has infinitely many objects by assumption.

Since any two neighboring entries in $\mathcal{F}$ are roots, applying (R1) again yields that all entries $\mathcal{F}_{i,j}$, $i<j$ are positive.
}
\end{proof}

\begin{Corollary}
For fixed $p$, the set of Cartan graphs of type $\Lambda_p$ of rank two with infinitely many 
vertices permitting a root system is parametrized by locally finite $p$-angulations of infinite strips.
\end{Corollary}


\end{document}